\documentclass[10pt,a4paper]{article}
\usepackage[british]{babel}
\usepackage{amssymb,amsmath,amsthm}
\usepackage{cite}
\usepackage{graphicx}
\usepackage[bookmarks=true, bookmarksopen=true, bookmarksopenlevel=4]{hyperref}
\usepackage{color}
\usepackage{enumerate}
\usepackage{pgf,tikz}
\usepackage{soul}
\usetikzlibrary{decorations.pathreplacing,matrix,calc,arrows,shapes.geometric}

\hypersetup{pdftitle={A space-time Trefftz discontinuous Galerkin method for the acoustic wave equation in first-order formulation},
  pdfauthor={A. Moiola, I. Perugia},
  colorlinks=true, linkcolor=myblue,  citecolor=myblue, filecolor=myblue,   urlcolor=myblue,}

\allowdisplaybreaks[4]

\allowdisplaybreaks[4]
\definecolor{myblue}{rgb}{0,0,0.6}         
\definecolor{gray}{rgb}{0.5,0.5,0.5}

\definecolor{amcol}{rgb}{0.8,0,0}
\newcommand{\am}[1]{{\color{amcol}{#1}}}

\newcommand*{\der}[2]{\frac{\partial #1}{\partial #2}}
\newcommand{\di}{\,\mathrm{d}}                              
\newcommand{\iin}{\;\text{in}\;}
\newcommand{\oon}{\;\text{on}\;}
\newcommand{\deO}{{\partial\Omega}}
\newcommand*{\conj}[1]{\overline{#1}}

\newcommand*{\N}[1]{\left\|#1\right\|}
\newcommand*{\abs}[1]{\left|#1\right|}
\newcommand{\Tnorm}[1]{|||#1|||}
\newcommand*{\jmp}[1]{[\![#1]\!]}                     
\newcommand*{\mvl}[1]{\{\!\!\{#1\}\!\!\}}             

\newcommand{\curl} {\mathop{\rm curl}\nolimits}
\newcommand{\dive} {\mathop{\rm div}\nolimits}
\def\div{\mathop{\rm div}\nolimits}

\DeclareMathOperator{\diam}{diam} %

\DeclareMathOperator{\spn}{span}

\newcommand{\uu}[1]{\hbox{\boldmath$#1$}}            
\newcommand{\Uu}[1]{{\mathbf{#1}}}                   

\newcommand{\IC}{\mathbb{C}}

\newcommand{\IN}{\mathbb{N}}\newcommand{\IP}{\mathbb{P}}
\newcommand{\IR}{\mathbb{R}}
\newcommand{\IT}{\mathbb{T}}\newcommand{\IU}{\mathbb{U}}
\newcommand{\IW}{\mathbb{W}}

\newcommand{\ba}{{\Uu a}}\newcommand{\bb}{{\Uu b}}
\newcommand{\bd}{{\Uu d}}\newcommand{\be}{{\Uu e}}

\newcommand{\bn}{{\Uu n}}

\newcommand{\bu}{{\Uu u}}
\newcommand{\bx}{{\Uu x}}
\newcommand{\by}{{\Uu y}}

\newcommand{\bF}{{\Uu F}}

\newcommand{\bM}{{\Uu M}}\newcommand{\bN}{{\Uu N}}

\newcommand{\bT}{{\Uu T}}
\newcommand{\bV}{{\Uu V}}\newcommand{\bX}{{\Uu X}}
\newcommand{\bY}{{\Uu Y}}

\newcommand{\bzeta}{{\uu\zeta}}         
        
\newcommand{\bsigma}{{\boldsymbol \sigma}}\newcommand{\btau}{{\boldsymbol\tau}}

\newcommand{\bxi}{{\boldsymbol\xi}}	

\newcommand{\bPhi}{{\boldsymbol \Phi}}      \newcommand{\bPsi}{{\uu \Psi}}

           \newcommand{\bzero}{\Uu{0}}

\newcommand{\calA}{{\mathcal A}}

\newcommand{\calE}{{\mathcal E}}\newcommand{\calF}{{\mathcal F}}

\newcommand{\calJ}{{\mathcal J}}
\newcommand{\calL}{{\mathcal L}}

\newcommand{\calO}{{\mathcal O}}

\newcommand{\calS}{{\mathcal S}}\newcommand{\calT}{{\mathcal T}}

\newcommand{\malpha}{{\boldsymbol{\alpha}}}\newcommand{\mbeta} {{\boldsymbol{\beta}}}

\newcommand{\Hdiv}{H(\div;\Omega)}

\newtheorem{theorem}{Theorem}[section]
\newtheorem{lemma}[theorem]{Lemma}

\newtheorem{prop}[theorem]{Proposition}
\newtheorem{proposition}[theorem]{Proposition}
\newtheorem{cor}[theorem]{Corollary}
\newtheorem{rem}[theorem]{Remark}
\newtheorem{remark}[theorem]{Remark}
\newtheorem{assum}[theorem]{Assumption}

\newcommand{\hh}{{\tt h}}
\newcommand{\tta}{{\tt a}}
\newcommand{\ttb}{{\tt b}}
\newcommand{\ttc}{{\tt c}}

\newcommand{\supp}{\operatorname{supp}}

\newcommand{\OO}{{(\Omega)}}
\newcommand{\Rn}{{(\IR^n)}}

\newcommand{\ee}{{\rm e}}
\newcommand{\ri}{{\rm i}}

\newcommand{\beq}{\begin{equation}}      \newcommand{\eeq}{\end{equation}}
\newcommand{\beqs}{\begin{equation*}}    \newcommand{\eeqs}{\end{equation*}}
\newcommand{\bit}{\begin{itemize}}       \newcommand{\eit}{\end{itemize}}
\newcommand{\ben}{\begin{enumerate}}     \newcommand{\een}{\end{enumerate}}
\newcommand{\bal}{\begin{align}}         \newcommand{\eal}{\end{align}}
\newcommand{\bals}{\begin{align*}}       \newcommand{\eals}{\end{align*}}
\newcommand{\bse}{\begin{subequations}}	 \newcommand{\ese}{\end{subequations}}
\newcommand{\bpr}{\begin{proposition}}   \newcommand{\epr}{\end{proposition}}
\newcommand{\bre}{\begin{remark}}        \newcommand{\ere}{\end{remark}}
\newcommand{\bpf}{\begin{proof}}         \newcommand{\epf}{\end{proof}}
\newcommand{\ble}{\begin{lemma}}         \newcommand{\ele}{\end{lemma}}
\newcommand{\bco}{\begin{corollary}}     \newcommand{\eco}{\end{corollary}}
\newcommand{\bex}{\begin{example}}       \newcommand{\eex}{\end{example}}
\newcommand{\bth}{\begin{theorem}}       \newcommand{\enth}{\end{theorem}}
            
\newcommand{\Fh}{\calF_h}
\newcommand{\Th}{{(\calT_h)}}

\newcommand{\deK}{{\partial K}}
\newcommand{\GD}{{\Gamma_D}}
\newcommand{\GN}{{\Gamma_N}}
\newcommand{\GR}{{\Gamma_R}}

\newcommand{\FT}{{\Fh^T}}
\newcommand{\FO}{{\Fh^0}}

\newcommand{\FD}{{\Fh^D}}
\newcommand{\FN}{{\Fh^N}}
\newcommand{\FR}{{\Fh^R}}

\newcommand{\Fspa}{{\Fh^{\mathrm{space}}}}
\newcommand{\Ftime}{{\Fh^{\mathrm{time}}}}

\newcommand{\deKspa}{{\partial^{\mathrm{space}}K}}
\newcommand{\deKtime}{{\partial^{\mathrm{time}}K}}

\newcommand{\hp}{_{hp}}
\newcommand\Vhp{v\hp}
\newcommand\Shp{\bsigma\hp}
\newcommand\hVhp{\widehat v\hp}
\newcommand\hShp{\widehat \bsigma\hp}

\newcommand{\Cstab}{M_{\mathrm{stab}}}
\newcommand{\tCstab}{\widetilde{M}_{\mathrm{stab}}}
\newcommand*{\Norm}[1]{\left\|#1\right\|}
\newcommand{\DG}{_{\mathrm{DG}}}
\newcommand{\DGp}{_{\mathrm{DG^+}}}
\newcommand{\DGQ}{_{\mathrm{DG}(Q_\Sigma)}}
\newcommand{\DGQp}{_{\mathrm{DG}^+\!(Q_\Sigma)}}
\newcommand{\tht}{\vartheta}

\newcommand{\dersec}[2]{{\frac{\partial^2 #1}{\partial{#2}^2}}}
\newcommand{\LtQ}{{L^2(Q)}}

\newcommand{\dt}[1]{{\mathrm{d}_{#1}}}
\newcommand{\dH}[1]{{\mathrm{H}_{#1}}}
\newcommand{\bVp}{{\bV_p}}
\newcommand{\ITp}{{\IT_p}}

\newcommand{\IWp}{{\IW_p}}

\definecolor{ipcol}{rgb}{0,0,0.9}

\usepackage{fancyhdr}
\renewcommand{\am}[1]{{#1}}

\title{A space--time Trefftz discontinuous Galerkin method for the \am{acoustic wave equation in} first-order \am{formulation}}

\author{Andrea Moiola\thanks{Department of Mathematics and Statistics, University of Reading, Whiteknights PO Box 220, Reading RG6 6AX, UK.
Current address: Department of Mathematics, University of  Pavia, 27100 Pavia, Italy (\texttt{andrea.moiola@unipv.it})},
Ilaria Perugia\thanks{Faculty of Mathematics, University of Vienna,
  1090 Vienna, Austria, and Department of Mathematics, University of
  Pavia, 27100 Pavia, Italy (\texttt{ilaria.perugia@univie.ac.at})}, 
}

\date{\today}

\begin{document}

\maketitle

\begin{abstract}
We introduce a space--time Trefftz discontinuous Galerkin method for
the first-order transient acoustic wave equations in arbitrary space
dimensions, extending the \am{one dimensional} scheme of 
Kretzschmar \textit{et al.} (2016, \textit{IMA J.\ Numer.\ Anal.}, 36, 1599--1635).
Test and trial discrete functions are space--time piecewise polynomial solutions of the wave equations.
We prove well-posedness and a priori error bounds in both skeleton-based 
and mesh-independent norms.
The space--time formulation corresponds to an implicit time-stepping scheme, if posed on meshes partitioned in time slabs, or to an explicit scheme, if posed on ``tent-pitched'' meshes.
We describe two Trefftz polynomial discrete spaces, introduce bases for them and prove optimal, high-order $h$-convergence bounds.


\end{abstract}
 
\section{Introduction}
%
Standard finite element methods seek to approximate a particular solution of a partial differential equation (PDE) by piecewise polynomials.
To enhance accuracy and efficiency, in the case of linear homogeneous problems, a natural idea is to choose the approximating functions from a class of (piecewise) solutions of the same PDE: this is the idea at the heart of Trefftz methods, which are named after the seminal work \cite{Trefftz1926} of E.~Trefftz.
In the last decades Trefftz schemes have been used for several different linear, most often elliptic, PDEs; see e.g.\ \cite{Qin05,LLHC08}.
Trefftz methods turned out to be particularly effective, and popular, for wave-propagation problems in time-harmonic regime at medium and high frequencies, where the oscillatory nature of the solutions makes standard methods computationally too expensive; see the recent survey \cite{TrefftzSurvey} and references therein.

Much less work has been devoted to Trefftz methods for time-dependent (linear) wave phenomena, see in particular
\am{\cite{Maciag05,WTF14,PFT09,BGL2016,KSTW2014,EKSW15,EKSTWtransparent,KretzschmarPhD,SpaceTimeTDG,LiK16} for numerical results and \cite{BGL2016,EKSW15,KretzschmarPhD,SpaceTimeTDG} for stability and convergence analyses.
Problems in one space dimension have been considered in \cite{KSTW2014,SpaceTimeTDG,PFT09}, while the other references studied two- and three-dimensional cases.}
While Trefftz methods for time-harmonic problems require non-polynomial basis functions, when used to discretise transient wave problems they admit special space--time polynomials as discrete functions.
This feature prevents excessive ill conditioning, which notoriously haunts time-harmonic Trefftz schemes (see \cite[\S4]{TrefftzSurvey}).
The earliest Trefftz methods for time-domain wave problems were proposed by A.\ Maci{\c{a}}g in \cite{Maciag05} and subsequent articles; these schemes are sorts of ``spectral'' Trefftz methods, in the sense that a single space--time element is used.
In \cite{PFT09,WTF14} Trefftz methods posed on triangulations of the space--time domain were introduced for the (second-order) acoustic wave equation; inter-element continuity 
of the solution is enforced by Lagrange multipliers. 
A Trefftz-interior penalty formulation is introduced and studied in~\cite{BGL2016}.
Another Trefftz discontinuous Galerkin (DG) formulation for time-dependent electromagnetic problems formulated as first-order systems has been proposed in \cite{KSTW2014} and analysed in \cite{SpaceTimeTDG} in one space dimension; it has been extended to full three-dimensional Maxwell equations in \cite{EKSW15,EKSTWtransparent,KretzschmarPhD}.

We mention here that, independently of the Trefftz approach, space--time finite elements for 
linear wave propagation problems, originally introduced in \cite{HughesHulbert1988} (see also~\cite{French1993,Jo93}), have been used in combination with DG formulations e.g.\ in~\cite{FalkRichter1999,CostanzoHuang05,MoRi05} and, more recently, in~\cite{DFW2016,GMS15,GSW16,Lilienthal:2014fs}.

In this paper we extend the Trefftz-DG method of \cite{SpaceTimeTDG} to initial boundary value problems for the acoustic wave equations posed on Lipschitz polytopes in arbitrary dimensions.
We write the acoustic wave problem as a first-order system, as it is originally derived from the linearised Euler equations, \cite[p.~14]{COR98}; we consider piecewise-constant wave speed, Dirichlet, Neumann and impedance boundary conditions.
The main focus of this paper is on the \textit{a priori} error analysis of the 
Trefftz--DG scheme.

The DG formulation proposed can be understood as the translation to time-domain of the Trefftz-DG formulation for the Helmholtz equation of \cite{PVersion}, which in turn is a generalisation of the Ultra Weak Variational Formulation (UWVF) of \cite{CED98}.
The DG numerical fluxes are upwind in time and centred with a special jump penalisation in space.
Under a suitable choice of the numerical flux coefficients, combining the proposed formulation with standard discrete spaces and complementing it with suitable volume terms, one recovers the DG formulation of \cite{MoRi05}, cf.\ Remark~\ref{rem:MR} below.
The Trefftz formulation for Maxwell's equations of 
\cite{KSTW2014,EKSW15,EKSTWtransparent,KretzschmarPhD} corresponds to the ``unpenalised'' version of that one proposed here (the numerical experiments in \cite[\S7.5]{SpaceTimeTDG} show that the numerical error depends very mildly on the penalisation parameters).


We first describe the IBVP under consideration in \S\ref{s:IBVP}, the assumptions on the mesh in  \S\ref{s:Mesh} and the Trefftz--DG formulation in \S\ref{s:TDG}.
Following the thread of \cite{PVersion,SpaceTimeTDG}, in \S\ref{s:WellP} and \S\ref{s:Energy} we prove that the scheme is well-posed, quasi-optimal, dissipative (quantifying dissipation using the jumps of the discrete solution), and derive error estimates for some traces of the solution on the mesh skeleton.
In \S\ref{s:MeshIndependent} we investigate how to control the Trefftz-DG error in a mesh-independent norm: 
after setting up a general duality framework in \S\ref{s:Duality},
we prove error bounds in $L^2(Q)$ norm ($Q$ being the space--time computational domain) under some restrictive assumptions on the mesh in \S\ref{s:Stability},
and in a weaker Sobolev norm in \S\ref{s:StabilityX} under different assumptions.

The analysis carried out in \S\ref{s:Analysis} holds for any choice of discrete Trefftz spaces.
In \S\ref{s:Spaces} we describe two different polynomial Trefftz spaces: one, denoted $\IT_p\Th$, in \S\ref{s:Tp} for general IBVPs for the first-order acoustic wave equations,
and one, denoted $\IW_p\Th$, in \S\ref{s:Sp} for IBVPs that are obtained from second-order problems.
For both discrete spaces we introduce simple bases and prove approximation estimates, which lead to fully explicit,  high-order (in the meshwidth $h$), optimal-in-$h$ convergence estimates for the Trefftz-DG method; see Theorems~\ref{thm:ConvergenceTp} and \ref{thm:ConvergenceSp}.
Estimates ensuring convergence with respect to the polynomial degree $p$, such as those proved in \cite[\S5.3.2]{SpaceTimeTDG} for one space dimension, 
are still elusive in the general case; the same situation occurs in~\cite{BGL2016}.

\am{The analysis differs from that of \cite{SpaceTimeTDG} in several respects: 
we consider higher-dimensional problems (which is the most fundamental difference),
space-like element faces not necessarily perpendicular to the time axis,
error bounds in mesh-independent norms other than $L^2$ 
(since bounds in $L^2$ norm do not seem possible in this generality),
we use different techniques to prove approximation properties of Trefftz polynomials (restricted to $h$-convergence only).
We expect that all results presented here, except possibly those of \S\ref{s:StabilityX} on error bounds in mesh-independent norm in the presence of time-like faces, can be extended to the case of Maxwell's equations in three space dimensions in a straightforward way.
}

Comparing against the Trefftz scheme of \cite{BGL2016} which is of interior penalty type, our error analysis does not use inverse estimates for polynomials, 
thus the analysis holds for any discrete Trefftz space (including non-polynomial ones, cf.\ Remark~\ref{rem:NonPolyTrefftz}) and the numerical flux parameters in the definition of the 
formulation are more easily determined (i.e.\ no parameter has to be ``large enough'').

One of the strengths of the Trefftz-DG method compared to non-Trefftz schemes is the much better asymptotic behaviour in terms of accuracy per number of degrees of freedom; 
this has already been described in details and demonstrated numerically in 
\cite{SpaceTimeTDG}.
More precisely, for a problem in $n$ space dimensions, the Trefftz approach allows to reduce
from $\calO(p^{n+1})$ to $\calO(p^n)$
the dimension of local space--time approximating spaces with effective $h$-approximation order
$p$.
%
A further advantage is that Trefftz schemes require quadrature to be performed on the mesh skeleton only, reducing the computational effort associated with the linear system assembly.

The Trefftz-DG formulation and its analysis admit the use of very general space--time meshes and discrete spaces, allowing local time-stepping, $hp$-refinement and interfaces not aligned to the space--time axes.
If the mesh elements can be collected in time slabs, the Trefftz-DG linear system is block-triangular, each block corresponding to a slab, thus its solution is completely analogue to an  unconditionally stable, implicit time-stepping.
One can also design a mesh in such a way that the Trefftz-DG system can be solved in an explicit fashion: this is the idea of ``{tent-pitched}'' meshes, 
see \cite{EGSU05,FalkRichter1999,GMS15,GSW16,UnSh02,APH06} and the comments in \S\ref{s:TDG} and \ref{s:Stability} below.
While the implementation of an explicit time-stepping method on a tent-pitched mesh might be quite cumbersome, the combination with a Trefftz discretisation can make it simpler as no volume quadrature on complicated shapes (the ``tents'') is needed, cf.~\S\ref{s:Stability}.

The description of the Trefftz-DG method and part of the analysis of \S\ref{s:Analysis} already appeared in the conference paper~\cite{CEDYA}.


\section{The initial boundary value problem}\label{s:IBVP}

We consider an initial boundary value problem (IBVP) posed on a space--time domain $Q=\Omega\times I$, where $\Omega\subset\IR^n$ is an open, bounded, Lipschitz polytope with outward unit normal $\bn_\Omega^x$, $n\in\IN$ and $I=(0,T)$, $T>0$.
The boundary of $\Omega$ is divided in three parts, with mutually
disjoint interiors, denoted $\GD$, $\GN$ and $\GR$, corresponding to
Dirichlet, Neumann and Robin boundary conditions, respectively; one or two of them may be empty.
The first-order acoustic wave IBVP reads as
\begin{align}\label{eq:IBVP}
\left\{\begin{aligned}
&\nabla v+\der{\bsigma}t = \bzero &&\iin Q,\\
&\nabla\cdot\bsigma+c^{-2}\der{v}t = 0 &&\iin Q,
\\
&v(\cdot,0)=v_0, \quad \bsigma(\cdot,0)=\bsigma_0 &&\oon \Omega,
\\
&v=g_D, &&\oon \GD\times [0,T],
\\
&\bsigma\cdot\bn_\Omega^x=g_N, &&\oon \GN\times [0,T],
\\
&\frac\tht c v-\bsigma\cdot \bn_\Omega^x=g_R , &&\oon \GR\times [0,T].
\end{aligned}
\right.
\end{align}
Here $v_0,\bsigma_0,g_D,g_N,g_R$ are the problem data;
$c\ge c_0>0$ is the  wave speed, which is assumed to be piecewise constant and independent of $t$;
$\tht\in L^\infty(\GR\times[0,T])$ is an impedance parameter, which is assumed to be uniformly positive.
The gradient $\nabla$ and divergence $\nabla\cdot$ operators are meant
in the space variable $\bx$ only \am{and $\bn_\Omega^x$ is the outward-pointing unit normal vector on $\deO\times[0,T]$}.
If \eqref{eq:IBVP} is obtained from the linearisation of Euler's equations as in \cite[p.~14]{COR98}, $v$ and $\bsigma$ represent the small perturbations 
of pressure and velocity fields, respectively, around a static state
\am{$v_{\mathrm{eq}}=$ constant and $\bsigma_{\mathrm{eq}}=\bzero$}.


If the initial condition $\bsigma_0$ is the gradient of some scalar field, namely $\bsigma_0=-\nabla U_0$,
then, setting $v=\der Ut$ and $\bsigma=-\nabla U$, IBVP \eqref{eq:IBVP} is equivalent to the following one for the second-order scalar wave equation 
\begin{align}\label{eq:IBVP_U}
\left\{\begin{aligned}
&-\Delta U+c^{-2}\der{{^2}}{t^2}U=0&&\iin Q,
\\
&\der U t(\cdot,0)=v_0, \quad U(\cdot,0)=U_0 &&\oon \Omega,
\\
&\der Ut=g_D, &&\oon \GD\times [0,T],
\\
&-\bn_\Omega^x\cdot\nabla U=g_N, &&\oon \GN\times [0,T],
\\
&\frac\tht c \der Ut+\bn_\Omega^x\cdot\nabla U=g_R , &&\oon \GR\times [0,T].
\end{aligned}
\right.
\end{align}
The conditions on $\GR\times[0,T]$ differ from the standard Robin boundary conditions $\vartheta U+\bn_\Omega^x\cdot\nabla U=g_R$ for the wave equation; the ones we use are called ``impedance boundary conditions'' in \cite[eq.~(37)]{MoRi05}, 
``dynamic boundary conditions'' in \cite[\S1.5]{Say16} and include low-order absorbing conditions.

\begin{rem}\label{rem:IBVP1toIBVP2}
If the vector initial datum  $\bsigma_0$ belongs to $\Hdiv$, the solution of IBVP~\eqref{eq:IBVP} can be reduced to that of a second-order problem in the form \eqref{eq:IBVP_U} using 
a Hodge--Helmholtz decomposition.
We first define $U_0\in H^1\OO$ to be a solution of the Laplace--Neumann problem
$$
\begin{cases}
-\Delta U_0= \nabla\cdot\bsigma_0 &\iin \Omega,\\
-\bn_\Omega^x\cdot\nabla U_0 =\bn_\Omega^x\cdot\bsigma_0&\oon \deO,
\end{cases}
$$
and fix $\bPsi_0:=\bsigma_0+\nabla U_0$.
Then, if $U$ is the solution of \eqref{eq:IBVP_U}, where $U_0$ is used as initial condition, the pair $(\der Ut,\Psi_0-\nabla U)$ is solution of \eqref{eq:IBVP} (where $\bPsi_0$ is independent of $t$).
Thus the solution of a IBVP for the first-order wave equations with initial datum  $\bsigma_0$ in $\Hdiv$ can always be written as the sum of a term obtained from the first-order derivatives of the solution $U$ of a IBVP for the second-order wave equation, and a divergence-free, time-independent term $\bPsi_0$.
(In Remark~\ref{rem:TvsS} we derive a similar decomposition for spaces of polynomial solutions of the wave equation.)
\end{rem}

\begin{remark} 
The case of an inhomogeneous IBVP (namely when one or both the PDEs in \eqref{eq:IBVP} have a non-zero right-hand side, 
$\nabla v+\der{\bsigma}t = \bPhi$ or $\nabla\cdot\bsigma+c^{-2}\der{v}t = \psi$) can be reduced to a homogeneous one by extending the source terms 
$\bPhi$ and $\psi$ to $\IR^n\times[0,T]$ (e.g.\ by zero) and constructing a particular solution with Duhamel's principle; see \cite[Sect. 2.4.2]{EVA02}.
\end{remark}

\section{Space--time mesh and notation}\label{s:Mesh}

We partition the space--time domain $Q$ with a finite element mesh $\calT_h$.
We assume that all its elements are Lipschitz polytopes, so that each internal face $F=\partial K_1\cap \partial K_2$, for $K_1,K_2\in\calT_h$, with positive $n$-dimensional measure, is a subset of a hyperplane: 
$$F\subset\Pi_F:=\big\{(\bx,t)\in \IR^{n+1}:\; \bx\cdot\bn^x_F+t\, n^t_F=C_F\big\},$$ 
where $(\bn^x_F,n^t_F)$ is a unit vector in $\IR^{n+1}$ and
$C_F\in\IR$. 
We make the following assumption:
\begin{align}
&\text{on an internal face } 
F=\partial K_1\cap \partial K_2 
\text{, either}
\nonumber\\
&\begin{cases}
c|\bn_F^x|< n_F^t & \text{and $F$ is called ``space-like'' face, or}\\
n_F^t=0 & \text{and $F$ is called ``time-like'' face.}
\end{cases}
\label{eq:HorVerFaces}
\end{align}
(See Remark \ref{rem:FullyUnstructured} for the extension to more general meshes.)
On space-like faces, by convention, we choose $n^t_F>0$, i.e.\ the unit normal vector $(\bn^x_F,n^t_F)$ points towards future time.
Intuitively, space-like faces are hypersurfaces lying under the characteristic cones and on which initial conditions might be imposed, while time-like faces are propagations in time of the faces of a mesh in space only.
We use the following notation:
\begin{align*}
\Fh&:=\bigcup\nolimits_{K\in\calT_h}\partial K \quad \text{(the mesh skeleton)},\\
\Fspa&:= \text{the union of the internal space-like faces,}\\
\Ftime&:= \text{the union of the internal time-like faces,}\\
\FO&:=\Omega\times\{0\},\hspace{12mm} \FT:=\Omega\times\{T\},\\ 
\FD&:=\GD\times[0,T],\qquad
\FN:=\GN\times[0,T],\qquad
\FR:=\GR\times[0,T]. 
\end{align*}
(We will consider more specific meshes either with $\Ftime=\emptyset$ in \S\ref{s:Stability} or with $(\bn^x_F,n^t_F)=(\bzero,1)$ on $\Fspa$ in \S\ref{s:StabilityX}.)
We denote the outward-pointing unit normal vector on $\deK$ by
$(\bn^x_K,n^t_K)$. 
\am{We report in Figure~\ref{fig:stmesh} an example of a $(1+1)$--dimensional mesh.}

{
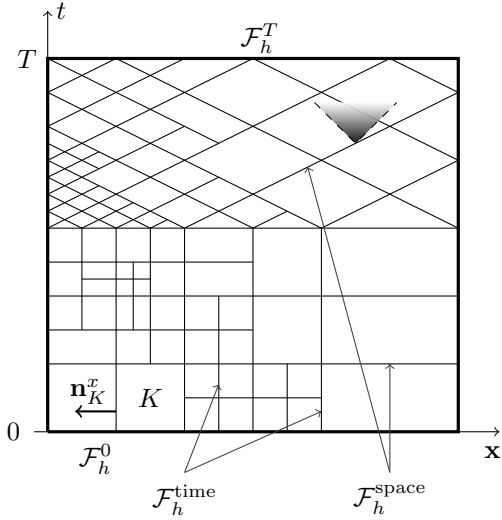
\begin{figure}
\centering
\begin{tikzpicture}[scale=0.9]
\draw [->] (0,-0.2)--(0,3.2);\draw(0.2,3.2)node{$t$};
\draw [->] (-0.2,-3) -- (6.6,-3);\draw(6.5,-3.3)node{$\bx$}; 
\draw [very thick] (0,-3)--(6,-3)--(6,2.5)--(0,2.5)--(0,-3);
\draw (0,0)--(5,2.5)--(6,2)--(2,0)--(0,1)--(3,2.5)--(6,1)--(4,0)--(0,2)--(1,2.5)--(6,0);
\draw (3.5,0.25)--(3,0)--(0,1.5)--(2,2.5);
\draw (2.5,0.75)--(1,0)--(0,0.5)--(2,1.5); 
\draw (0,2.5)--(2.5,1.25);
\draw (1.75,0.125)--(1.5,0)--(0,0.75)--(0.75,1.125);
\draw (1.25,0.375)--(0.5,0)--(0,0.25)--(0.75,0.625);
\draw (0.75,0.875)--(0,1.25);
\draw (0.5,-.75)--(1.5,-.75);
\draw (1.25,-.5)--(1.25,-1.5);
\draw (0,0)--(6,0);
\draw (0,-1)--(6,-1);
\draw (0,-2)--(6,-2);
\draw (1,0)--(1,-3);
\draw (2,0)--(2,-3);
\draw (3,0)--(3,-3);
\draw (4,0)--(4,-3);
\draw (0.5,0)--(0.5,-1.5);
\draw (1.5,0)--(1.5,-2);
\draw (0,-0.5)--(3,-0.5);
\draw (0,-1.5)--(3,-1.5);
\draw (2.5,-1)--(2.5,-3);
\draw (3.5,-2)--(3.5,-3);
\draw (2,-2.5)--(4,-2.5);
\draw(.7,-3.4)node{$\FO$}; 
\draw(3.1,2.8)node{$\FT$}; 
\draw(-0.5,-3)node{$0$}; 
\draw(-0.3,2.5)node{$T$}; 
\draw(1.5,-2.5)node{$K$};
\draw [->,thick] (1,-2.7) -- (0.4,-2.7);
\draw(.6,-2.4)node{$\bn_K^x$};
\draw(2,-4)node{$\Ftime$};
\draw [->,darkgray] (2,-3.6) -- (2.5,-2.3);
\draw [->,darkgray] (2,-3.6) -- (4,-2.7);
\draw(5,-4)node{$\Fspa$};
\draw [->,darkgray] (5,-3.6) -- (5,-2);
\draw [->,darkgray] (5,-3.6) -- (3.8,.9);
\shade[bottom color=black, top color=white](3.9,1.85)--(4.5,1.25)--(5.1,1.85);
\draw [dashed] (3.9,1.85)--(4.5,1.25)--(5.1,1.85);
\end{tikzpicture}
\caption{\am{Example of a $(1+1)$--dimensional mesh.
The shaded triangle represent the characteristic cone of a point of a space-like face: the cone lies above the face.}}
\label{fig:stmesh}
\end{figure}
}

For piecewise-continuous scalar ($w$) and vector ($\btau$) fields, we
define averages $\mvl\cdot$, space normal jumps $\jmp{\cdot}_\bN$ and
time (full) jumps $\jmp{\cdot}_t$ on internal faces in the standard DG notation:
on $F=\deK_1\cap\deK_2$, $K_1,K_2\in\calT_h$,
\begin{align*}
\mvl{w}&:=\frac{w_{|_{K_1}}+w_{|_{K_2}}}2,\qquad
\mvl{\btau}:=\frac{\btau_{|_{K_1}}+\btau_{|_{K_2}}}2, \\
\jmp{w}_\bN&:= w_{|_{K_1}}\bn_{K_1}^x+w_{|_{K_2}} \bn_{K_2}^x,\\
\jmp{\btau}_\bN&:= \btau_{|_{K_1}}\cdot\bn_{K_1}^x+\btau_{|_{K_2}} \cdot\bn_{K_2}^x,\\
\jmp{w}_t&:= w_{|_{K_1}}n_{K_1}^t+w_{|_{K_2}} n_{K_2}^t=(w^--w^+)n^t_F,\\
\jmp{\btau}_t&:= \btau_{|_{K_1}} n_{K_1}^t+\btau_{|_{K_2}} n_{K_2}^t= (\btau^--\btau^+)n^t_F.
\end{align*}
The time jumps $\jmp{\cdot}_t$ are different from zero on space-like faces only. 
Here we have denoted by $w^-$ and $w^+$ the traces of the function $w$ from the adjacent elements at lower and higher times, respectively (and similarly for $\btau^\pm$).
We use the notation $\jmp{\cdot}_\bN$ to recall that $\jmp{w}_\bN$ is a vector field parallel to $\bn^x_F$ and $\jmp{\btau}_\bN$ is the jump of the normal component of $\btau$ only.
We recall also that in \am{these} formulas $|n^t_K|,|\bn^x_K|\le1$, and that
one of the two might be zero; in particular, as already pointed out, $\jmp{w}_t=0$ and $\jmp{\btau}_t=\bzero$ on $\Ftime$.
The following identities can easily be shown:
%
\begin{align}\nonumber
&w^-\jmp{w}_t-\frac12 \jmp{w^2}_t=\frac1{2n^t_F}\jmp{w}_t^2
&&{\oon\Fspa},\\
&\btau^-\jmp{\btau}_t-\frac12 \jmp{\btau^2}_t=\frac1{2n^t_F}|\jmp{\btau}_t|^2
&&{\oon\Fspa},\label{eq:JumpIds}\\
&w^-\jmp{\btau}_\bN+\btau^-\cdot\jmp{w}_\bN-\jmp{w\btau}_\bN   
=\frac1{n^t_F}\jmp{w}_t\jmp{\btau}_\bN
&&{\oon\Fspa},\nonumber\\
&\mvl{w}\jmp{\btau }_\bN+\mvl{\btau}\cdot\jmp{ w }_\bN=\jmp{w\btau}_\bN
&&{\oon\Fspa\cup\Ftime}.
\nonumber
\end{align}
We assume that the mesh $\calT_h$ is chosen  so that the wave speed $c$ is constant in each element; as $c$ is independent of time, it may jump only across faces in $\Ftime$.
The assumptions on the mesh made so far are sufficient for part of the error analysis of \S\ref{s:Analysis}; in order to prove error bounds in mesh-independent norms and orders of 
convergence, additional assumptions on the shape of the elements will be specified in the following (see, in  particular, Corollaries~\ref{cor:NoTimeLike} and~\ref{cor:Hminus1}, 
Theorems~\ref{thm:ConvergenceTp} and \ref{thm:ConvergenceSp}).

Finally, we define 
local and global \emph{Trefftz spaces}:
\begin{align*} 
\bT(K):=&\Big\{(w ,\btau )\in L^2(K)^{1+n} \text{ s.t. }
\btau_{|_{\deK}}\in L^2(\deK)^n,
\der w t, \nabla\cdot \btau \in  L^2(K), \\&\qquad
\der \btau t, \nabla w \in L^2(K)^n,\quad 
\underbrace{\nabla w+\der{\btau}t = \bzero, \;
\nabla\cdot\btau+c^{-2}\der{w}t = 0}_{\text{(Trefftz property)}}
 \Big\}\qquad\forall K\in \calT_h,
\\
\bT(\calT_h):=&\Big\{(w ,\btau )\in L^2(Q)^{1+n},\text{ s.t. }
(w_{|_K},\btau_{|_K})\in \bT(K) \;\forall K\in \calT_h\Big\}.
\end{align*}
The solution $(v,\bsigma)$ of IBVP \eqref{eq:IBVP} is assumed to
belong to $\bT\Th$\am{; as it clearly satisfies the Trefftz property,
this is an assumption on its smoothness (in general the solution of the IBVP \eqref{eq:IBVP} belongs to the graph space of the PDE, so for example $\der vt+\nabla\cdot\boldsymbol\sigma\in L^2(K)$ for all elements $K$, but not necessarily $\der vt\in L^2(K)$ and $\nabla\cdot\boldsymbol\sigma\in L^2(K)$;
similarly, the trace of $\boldsymbol\sigma$ on $\partial K$ might not be square-integrable)}.

\section{The Trefftz-discontinuous Galerkin method}\label{s:TDG}

To obtain the DG formulation, we multiply the first two equations of \eqref{eq:IBVP} with test fields $\btau$ and $w$ and integrate by parts on a single mesh element $K\in\calT_h$:
\begin{align}
\label{eq:Elemental0}
&-\int_K\bigg(v
\Big(
\nabla\cdot\btau+c^{-2}\der wt 
\Big) +\bsigma\cdot\Big(
\der\btau t+\nabla w  
\Big)\bigg)\di V
\\
&+\int_{\partial K} \bigg((v\,\btau + \bsigma\, w)\cdot\bn_K^x
+\Big(\bsigma\cdot\btau + c^{-2} v\,w\Big)\,n_K^t\bigg)\di S
=0.
\nonumber
\end{align}
We look for a discrete solution $(v\hp,\bsigma\hp)$ approximating
$(v,\bsigma)$ in a finite-dimensional (arbitrary, at this stage)
Trefftz space $\bVp \Th\subset\bT\Th$, where the subscript $p$ is related to the dimension of the local spaces. 
We take the test field $(w,\btau)$ in the same space $\bVp\Th$, thus \am{the volume integral over $K$ in \eqref{eq:Elemental0} vanishes}.
The traces of $v\hp$ and $\bsigma\hp$ on the mesh skeleton are approximated by the (single-valued) \emph{numerical fluxes} $\hVhp$ and $\hShp$, so that \eqref{eq:Elemental0} is rewritten as:
\begin{align}\label{eq:Elemental1}
\int_{\partial K} 
\hVhp\Big(\btau \cdot\bn_K^x+\frac{w}{c^2} n_K^t\Big)
+\hShp\cdot\big(w\bn_K^x+\btau\,n_K^t\big)
\di S=0.
\end{align}
We choose to define the numerical fluxes as:
\begin{align*}
\hVhp :=\begin{cases}
\Vhp^- \\
\Vhp \\
v_0\\ 
\mvl{\Vhp}+\beta \jmp{\Shp}_\bN \\
g_D\\
v\hp+\beta(\bsigma\hp\cdot\bn_\Omega^x-g_N)\\
(1-\delta)v\hp+\frac{\delta c}\tht (\bsigma\hp\cdot\bn_\Omega^x+g_R)
\end{cases}\hspace{-4mm}
\hShp:=
\begin{cases}
\Shp^- & \oon \Fspa,\\
\Shp &\oon \FT,\\
\bsigma_0 &\oon \FO,\\
\mvl{\Shp}+\alpha \jmp{\Vhp}_\bN \hspace{-2mm}&\oon \Ftime,\\
\Shp+\alpha(v\hp-g_D)\bn_\Omega^x\hspace{-2mm}&\oon \FD,\\
g_N\bn_\Omega^x
&\oon \FN,\\
(1-\delta)(\tht\frac {v\hp}c-g_R)\bn_\Omega^x+\delta\bsigma\hp&\oon \FR.
\end{cases}
\end{align*}
The parameters $\alpha\in L^\infty(\Ftime\cup\FD)$,  $\beta\in
L^\infty(\Ftime\cup\FN)$ and  
$\delta\in L^\infty(\FR)$ will be chosen depending on the mesh. 
They may be used to tune the method, e.g.\ to deal with locally-refined
meshes (see \S\ref{s:StabilityX} below).

These fluxes are \emph{consistent}, in the sense that they coincide
with the traces of the exact solution $(v,\bsigma)$ of the IBVP \eqref{eq:IBVP} if they are applied to $(v,\bsigma)$ itself, which satisfies the boundary conditions and has no jumps across mesh faces.
Moreover, the fluxes satisfy 
$\frac\tht c \hVhp-\hShp\cdot\bn_\Omega^x=g_R$ on $\FR$.
The numerical fluxes can be understood as upwind fluxes on the space-like faces and centred fluxes with jump penalisation on the time-like ones.

Summing the elemental TDG equation \eqref{eq:Elemental1} over the elements $K\in \calT_h$, with the fluxes defined above, we obtain the Trefftz-DG variational formulation:
\begin{align}
\text{Seek}&\; (\Vhp,\Shp)\in \bVp (\calT_h) \text{ such that }\;
\calA(\Vhp,\Shp; w ,\btau )=\ell( w ,\btau )\quad 
\forall ( w ,\btau )\in \bVp (\calT_h), \text{ where}
\nonumber\\
\calA(&\Vhp,\Shp; w ,\btau ):=
\int_{\Fspa}\big(c^{-2}\Vhp^-\jmp{w}_t+\Shp^-\cdot\jmp{\btau}_t+\Vhp^-\jmp{\btau}_\bN+\Shp^-\cdot\jmp{w}_\bN\big)\di S
\label{eq:TDG}
\\
&+\int_\FT (c^{-2}\Vhp  w +\Shp \cdot\btau )\di \bx
\nonumber\\
&+\int_\Ftime \big( \mvl{\Vhp}\jmp{\btau }_\bN+\mvl{\Shp}\cdot\jmp{ w }_\bN
+\alpha\jmp{\Vhp}_\bN\cdot\jmp{ w }_\bN+ \beta\jmp{\Shp}_\bN\jmp{\btau }_\bN
\big)\di S
\nonumber\\
&+\int_\FD \big(\bsigma\cdot\bn_\Omega^x\, w +\alpha \Vhp w   \big) \di S+\int_\FN\big(v\hp(\btau\cdot\bn_\Omega^x)
+\beta(\bsigma\cdot\bn_\Omega^x)(\btau\cdot\bn_\Omega^x)\big)\di S
\nonumber\\
&+\int_\FR\Big(
\frac{(1-\delta)\tht}c v\hp w+(1-\delta)v\hp(\btau\cdot\bn_\Omega^x)
+\delta(\bsigma\hp\cdot\bn_\Omega^x) w+\frac{\delta c}\tht(\bsigma\hp\cdot\bn_\Omega^x) (\btau\cdot\bn_\Omega^x)
\Big)\di S,
\nonumber\\
\ell( &w ,\btau ):=
\int_\FO ( c^{-2}v_0 w  +\bsigma_0\cdot \btau )\di \bx
\nonumber\\&
+\int_\FD g_D\big(\alpha  w -\btau\cdot\bn_\Omega^x\big)\di S
+\int_\FN g_N \big(\beta\,\btau\cdot\bn_\Omega^x-w\big)\di S
\nonumber\\&
+\int_\FR g_R \Big((1-\delta)w-\frac{\delta c}\tht\,\btau\cdot\bn_\Omega^x\Big)\di S.
\nonumber
\end{align}
Method \eqref{eq:TDG} appears as a  formulation over the whole space--time domain $Q$, which would lead to a global linear system coupling all elements.
However, if the mesh is suitably designed, the system matrix is block-triangular and the solution can be computed by solving a sequence of smaller local problem.
A first possibility is to partition the time interval $[0,T]$ into subintervals and solve sequentially for the corresponding time slabs $\Omega\times(t_{j-1},t_j)$, see \cite{EKSW15,KSTW2014}; this corresponds to an \emph{implicit} method and allows local mesh refinement.
\am{A slightly more complicated, but potentially much more efficient version is to solve for small patches of elements, localised in space and time, in the spirit of the \emph{semi-explicit} ``tent-pitching'' algorithm of \cite{EGSU05,FalkRichter1999,MoRi05,APH06}. 
If no time-like faces are present in the mesh, the solution can be found by solving a small system for each element, see \S\ref{s:Stability} below.
In the semi-explicit solution of the system arising from \eqref{eq:TDG}, assumption \eqref{eq:HorVerFaces} (equivalently, $\gamma<1$) plays the role of a CFL condition.}
If the same mesh is used, all these approaches are 
equivalent, in the sense that the discrete solutions $(\Vhp,\Shp)$ coincide.

\begin{remark}\label{rem:FullyUnstructured}
One could easily extend the formulation weakening assumption \eqref{eq:HorVerFaces} to allow more general time-like faces with $c|\bn^x_F|>n^t_F$, i.e.\ not aligned to the time-axis.
Choosing the numerical fluxes as above, one obtains a formulation similar to \eqref{eq:TDG} with additional terms on $\Ftime$ containing $\jmp{w}_t$ and $\jmp{\btau}_t$, which do not vanish in this setting.
It is then easy to prove the coercivity of the new bilinear form in a slight modification of the DG norm introduced below.
However, the bilinear form will contain the term $\int_\Ftime\mvl{\Shp}\cdot\jmp{\btau}_t\di S$, featuring the full jump of $\btau$ (as opposed to the normal jump only), which, in dimension $n>1$, is not easily controlled by the same DG norms.
\end{remark}

\begin{remark}\label{rem:MR}
Formulation \eqref{eq:TDG} can be seen in the framework of DG methods for general first-order hyperbolic systems developed in \cite{MoRi05}, which considers standard discontinuous piecewise-polynomial spaces.
The choice of the numerical fluxes on the interior faces correspond to the choice of a suitable decomposition of the block matrix 
$\mathsf M=(\begin{smallmatrix}n^t_K c^{-2} &(\bn^x_K)^\top\\ \bn^x_K & n^t_K \mathrm{Id}_n \end{smallmatrix})$ 
defined on $\deK$ for all $K\in \calT_h$.
Here $\mathrm{Id}_n$ is the identity matrix in $\IR^{n\times n}$, ${}^\top$ denotes vector transposition, and $\bn^x_K$ is thought as a column vector.
The choice we have made in this section corresponds to the decomposition:
\begin{center}
\begin{tabular}{l|l|l}
on & $\mathsf M^+ =$ & $\mathsf M^-=$ \\
\hline
$\partial K\cap\Fspa\cap\{n^t_K>0\}$&
$\mathsf M$&
$\bzero$
\\
$\partial K\cap\Fspa\cap\{n^t_K<0\}$&$\bzero$&
$\mathsf M$
\\
$\partial K\cap\Ftime$ 
&$(\begin{smallmatrix}\alpha &\frac12(\bn^x_K)^\top\\ 
\frac12\bn^x_K & \beta\bn^x_K\otimes \bn^x_K\end{smallmatrix})$&
$(\begin{smallmatrix}-\alpha &\frac12(\bn^x_K)^\top\\ 
\frac12\bn^x_K & -\beta\bn^x_K\otimes \bn^x_K\end{smallmatrix})$
\end{tabular} 
\end{center}
The conditions $\mathsf M^+ +\mathsf M^-=\mathsf M $,
$\ker(\mathsf M^+ -\mathsf M^-)=\ker(\mathsf M)$
and $\mathsf M^+_{|_{\deK_1}} +\mathsf M^-_{|_{\deK_2}}=\bzero$ on $\deK_1\cap\deK_2$
are automatically satisfied (recall that $n^t_K=0$ on $\Ftime$).
Moreover, $\mathsf M^+\ge0$ and $\mathsf M^-\le 0$ hold true if and only if $\alpha\beta\ge1/4$.
The boundary terms in \eqref{eq:TDG} and in \cite[\S6]{MoRi05} coincide (apart from a different sign convention) if our flux parameters and their boundary coefficients $Q$ and $\sigma$ are chosen so that 
$\alpha=\sigma$ on $\GD$, $\beta=1/\sigma$ on $\GN$,  $\delta=(1+Q)/2$ and $\tht/c=\sigma(1+Q)/(1-Q)$ on $\GR$.
\end{remark}

\section{A priori error analysis}\label{s:Analysis}

In this section we prove \textit{a priori} error bounds for the Trefftz-DG approximation of the solution to~\eqref{eq:IBVP}. 
After introducing mesh-dependent norms in Section~\ref{s:Assumptions}, we prove well-posedness and quasi-optimality of the Trefftz-DG formulation in Section~\ref{s:WellP} and error bounds on space-like mesh interfaces in Section~\ref{s:Energy}. 
Error bounds in mesh-independent norms are derived in Section~\ref{s:MeshIndependent}.

\subsection{Definitions and assumptions}\label{s:Assumptions}

We prove the well-posedness and the stability of the Trefftz-DG formulation \eqref{eq:TDG} under the assumption that the flux parameters $\alpha$, $\beta$ and $\delta$ are uniformly positive in their domains of definition and that $\Norm{\delta}_{L^\infty(\FR)}<1$.
We introduce a piecewise-constant function $\gamma$ defined on $\Fspa\cup\FO\cup\FT$, measuring how close to characteristic cones the space-like mesh faces are:
\begin{align}\label{eq:gamma}
\gamma:= \frac{c|\bn^x_F|}{n^t_F}\;\oon F
\subset\Fspa, \qquad \gamma:=0\;\oon \FO\cup\FT,
\end{align}
from which, recalling assumption \eqref{eq:HorVerFaces}, $\gamma\in[0,1)$ and 
\begin{align}\label{eq:JumpIneq}
&
\abs{\jmp{w}_\bN}\le \frac \gamma c\abs{\jmp{w}_t},\quad
\abs{\jmp{\btau}_\bN}\le \frac \gamma c\abs{\jmp{\btau}_t} \quad\oon \Fspa.
\end{align}
%
We define two mesh- and flux-dependent norms on $\bT\Th$:
\begin{align}\label{eq:DGnorm}
\Tnorm{(w ,\btau )}^2\DG&:=
\frac12 \Norm{\Big(\frac{1-\gamma}{n^t_F}\Big)^{1/2}c^{-1}\jmp{w }_t}_{L^2(\Fspa)}^2
+\frac12\Norm{\Big(\frac{1-\gamma}{n^t_F}\Big)^{1/2}\jmp{\btau }_t\rule{0pt}{3mm}}_{L^2(\Fspa)^n}^2 \\\nonumber
&\quad
+\frac12\Norm{c^{-1}w }^2_{L^2(\FO\cup\FT)} 
+\frac12\Norm{\btau \rule{0pt}{3mm}}^2_{L^2(\FO\cup\FT)^n} 
\\&\quad\nonumber
+\Norm{\alpha^{1/2}\jmp{w }_\bN}_{L^2(\Ftime)^n}^2
+\Norm{\beta^{1/2}\jmp{\btau }_\bN}_{L^2(\Ftime)}^2\\\nonumber
&\quad+\Norm{\alpha^{1/2}w }_{L^2(\FD)}^2
+\Norm{\beta^{1/2} \btau\cdot\bn_\Omega^x }_{L^2(\FN)}^2\\\nonumber
&\quad+\Norm{\Big(\frac{(1-\delta)\tht}c\Big)^{1/2} w }_{L^2(\FR)}^2
+\Norm{\Big(\frac{\delta c}\tht\Big)^{1/2} \btau\cdot\bn_\Omega^x }_{L^2(\FR)}^2;
\\
\nonumber
\Tnorm{(w ,\btau )}^2\DGp&:=
\Tnorm{(w ,\btau )}^2\DG
\\&\quad\nonumber
+2\Norm{\Big(\frac{n^t_F}{1-\gamma}\Big)^{1/2} c^{-1}w^-}_{L^2(\Fspa)}^2
+2\Norm{\Big(\frac{n^t_F}{1-\gamma}\Big)^{1/2}\btau^-}_{L^2(\Fspa)^n}^2  
\\&\quad\nonumber
+\Norm{
\beta^{-1/2}\mvl{w }}_{L^2(\Ftime)}^2
+\Norm{
\alpha^{-1/2}\mvl{\btau }}_{L^2(\Ftime)^n}^2
\\&\quad\nonumber
+\Norm{\alpha^{-1/2}\btau\cdot\bn_\Omega^x }_{L^2(\FD)}^2
+\Norm{\beta^{-1/2} w }_{L^2(\FN)}^2.
\end{align}
These are only \textit{seminorms} on broken Sobolev spaces defined on the mesh $\calT_h$, 
but are norms on $\bT\Th$: indeed $\Tnorm{(w,\btau)}=0$ for
$(w,\btau)\in\bT\Th$ implies that  $(w,\btau)$ is solution of the IBVP
\eqref{eq:IBVP} with zero initial and boundary conditions, so
$(w,\btau)=(0,\bzero)$ by the well-posedness of the IBVP itself (see
\cite[Lemma~4.1]{SpaceTimeTDG}).

\begin{rem}
The vector average $\mvl{\btau}$ in the definition of the $\Tnorm{\cdot}\DGp$ norm can be substituted by its normal component $\frac12(\btau_{|_{K_1}}\cdot\bn_{K_1}^x-\btau_{|_{K_2}}\cdot\bn_{K_2}^x)\bn_{K_1}^x$ on $\deK_1\cap\deK_2\cap\Ftime$.
The analysis carried out in the following is not altered by this change.
\end{rem}

\subsection{Well-posedness}\label{s:WellP}

We first note that for all Trefftz field $(w,\btau)\in\bT\Th$
\begin{equation}\label{eq:boundaryintegral}
\int_{\deK}
\bigg(w\btau\cdot\bn^x_K+\frac12\Big(c^{-2}{w^2}+|\btau|^2\Big)n^t_K\bigg)\di
  S=0\quad \forall K\in\calT_h,
\end{equation}
which follows from integration by parts and the Trefftz property.
Subtracting these terms from the bilinear form $\calA$, 
using the jump identities \eqref{eq:JumpIds}, the inequalities \eqref{eq:JumpIneq}, the definition of $\gamma$ in \eqref{eq:gamma}, and the weighted Cauchy--Schwarz inequality, we show that the form   $\calA$ is coercive in $\Tnorm{\cdot}\DG$ norm with unit  constant\am{. Indeed, for all $(w,\btau)\in \bT\Th$, we have}
\begin{align*}
\calA(w,\btau; w ,\btau )
\overset{\am{\eqref{eq:boundaryintegral}}}=&
\calA(w,\btau; w ,\btau )
-\sum_{K\in\calT_h}\int_{\deK}
\bigg(w\btau\cdot\bn^x_K+\frac{1}2\Big(c^{-2}{w^2}+|\btau|^2\Big)n^t_K\bigg)\di S
\\
\overset{\eqref{eq:TDG}}=&\int_{\Fspa}\big(c^{-2}w^-\jmp{w}_t+\btau^-\cdot\jmp{\btau}_t
+w^-\jmp{\btau}_\bN+\btau^-\cdot\jmp{w}_\bN
\\&\qquad
-\jmp{w\btau}_\bN-\frac12\jmp{c^{-2}w^2+|\btau|^2}_t\big)\di \bx
\\&
+\frac12\int_\FT (c^{-2}w^2  +\abs{\btau}^2 )\di \bx
+\frac12\int_\FO(c^{-2}w^2  +\abs{\btau}^2 )\di \bx
\\&
+\int_\Ftime \Big( \mvl{w}\jmp{\btau }_\bN+\mvl{\btau}\cdot\jmp{ w }_\bN
+\alpha|\jmp{w}_\bN|^2+ \beta\jmp{\btau}_\bN^2-\jmp{w\btau}_\bN\Big)\di S
\\&
+\int_\FD \alpha  w^2    \di S
+\int_\FN \beta\am{(\btau\cdot\bn_\Omega^x)^2}
\di S
\\&
+\int_\FR\Big(\frac{(1-\delta)\tht}c w^2
+\frac{\delta c}\tht(\btau\cdot\bn_\Omega^x)^2
\Big)\di S
\\
\overset{\eqref{eq:JumpIds}}=&
\int_{\Fspa}\bigg(
\frac1{2n^t_F}(c^{-2}\jmp{w}_t^2+|\jmp{\btau}_t|^2)+\frac1{n^t_F}\jmp{w}_t\jmp{\btau}_\bN
\bigg)\di \bx
\\&
+\frac12\Norm{c^{-1}w }^2_{L^2(\FO\cup\FT)} 
+\frac12\Norm{\btau \rule{0pt}{3mm}}^2_{L^2(\FO\cup\FT)^n}
\\&
+\Norm{\alpha^{1/2}\jmp{w}_\bN}^2_{L^2(\Ftime)}
+\Norm{\beta^{1/2}\jmp{\btau}_\bN}^2_{L^2(\Ftime)}
\\&
+\Norm{\alpha^{1/2}w }_{L^2(\FD)}^2
+\Norm{\beta^{1/2} \btau\cdot\bn_\Omega^x }_{L^2(\FN)}^2
\\&
+\Norm{\Big(\frac{(1-\delta)\tht}c\Big)^{1/2} w }_{L^2(\FR)}^2
+\Norm{\Big(\frac{\delta c}\tht\Big)^{1/2} \btau\cdot\bn_\Omega^x }_{L^2(\FR)}^2
\\
\overset{\eqref{eq:gamma},\eqref{eq:JumpIneq}}
\ge&
\Tnorm{(w,\btau)}\DG^2.
\end{align*}
In particular, if all the space-like faces are perpendicular to the time axis (i.e.\ $\bn_K^x=\bzero$ on them) then $\calA(w,\btau; w ,\btau )=\Tnorm{(w,\btau)}\DG^2$ for all $(w,\btau)\in \bT\Th$.
Using again the Cauchy--Schwarz inequality, the bounds on the jumps \eqref{eq:JumpIneq} and $\gamma<1$, we have the following continuity estimate for the bilinear form $\calA$: for all $(v,\bsigma),(w ,\btau )\in\bT\Th$
\begin{align}
\nonumber
&\abs{\calA(v,\bsigma; w ,\btau )}\le C_c\Tnorm{(v,\bsigma)}\DGp\Tnorm{(w,\btau)}\DG,
\quad \text{where}
\\&
C_c:=
\begin{cases}
2, &\text{if }\FR=\emptyset,\\
2\max\Big\{\Norm{\frac{1-\delta}\delta}_{L^\infty(\FR)}^{1/2},
\big\|\frac\delta{1-\delta}\big\|_{L^\infty(\FR)}^{1/2}\Big\}
&\text{if }\FR\ne\emptyset.
\end{cases}
\label{eq:Continuity}
\end{align}
Note that $C_c\ge2$ and that the minimal value $C_c=2$ is obtained for $\delta=1/2$.
Also the linear functional $\ell$ is continuous:
\begin{align*}
\abs{\ell(w,\btau)}
\le\Big(&2\Norm{c^{-1}v_0}_{L^2(\FO)}^2+2\Norm{\bsigma_0}_{L^2(\FO)}^2
+2\Norm{\alpha^{1/2}g_D}_{L^2(\FD)}^2
\\&
+2\Norm{\beta^{1/2}g_N}_{L^2(\FN)}^2
+
\Norm{(c/\tht)^{1/2}g_R}_{L^2(\FR)}^2
\Big)^{1/2} \Tnorm{(w,\btau)}\DGp;
\end{align*}
if $g_D=g_N=0$ (or the corresponding parts $\FD,\FN$ of the boundary are empty) then 
the $\Tnorm{(w,\btau)}\DGp$ norm at the right-hand side can be substituted by $\Tnorm{(w,\btau)}\DG$.

Combining coercivity and continuity, C\'ea's lemma gives well-posedness and quasi-op\-ti\-mal\-i\-ty of the Trefftz-DG formulation, irrespectively of the discrete Trefftz space $\bVp\Th$ chosen.

\begin{theorem}\label{th:QO}
The variational problem \eqref{eq:TDG} admits a unique solution $(v\hp,\bsigma\hp)\in \bVp\Th$.
It satisfies the error bound
\begin{align}\label{eq:QuasiOpt}
\Tnorm{(v-\Vhp,\bsigma-\Shp)}\DG\le  (1+C_c) \inf_{(w,\btau)\in\bVp\Th} \Tnorm{(v-w,\bsigma-\btau)}\DGp,
\end{align}
with $C_c$ as in \eqref{eq:Continuity}.
Moreover, if $g_D=g_N=0$ (or the corresponding parts $\FD,\FN$ of the boundary are empty) then
\begin{align*}
&\Tnorm{(v\hp,\bsigma\hp)}\DG\le
\bigg(2\Norm{c^{-1}
v_0}_{L^2(\FO)}^2+2\Norm{\bsigma_0}_{L^2(\FO)}^2
+\Norm{
(c/\tht)^{1/2}
g_R}_{L^2(\FR)}^2
\bigg)^{1/2}.
\end{align*}
\end{theorem}
Bound \eqref{eq:QuasiOpt} controls the Galerkin error in $\Tnorm{\cdot}\DG$ norm only;
we build on this result to bound stronger norms in two different ways.
First, in Section~\ref{s:Energy}, we use an energy argument to control
the norm of the trace of the error on space-like interfaces (see Proposition~\ref{prop:errorSigma}),
as bound~\eqref{eq:QuasiOpt} only provides control of the \emph{jumps} of the error on the same faces.
Second,
in Section~\ref{s:MeshIndependent} below, we bound the error in a mesh-independent norm with a duality-type argument inspired by \cite[Theorem~3.1]{MOW99} and based on an auxiliary problem. 



\subsection{Energy estimates and error bounds on space-like interfaces}
\label{s:Energy}

In this section, we introduce the energy functional for the wave equations, show that the Trefftz-DG method is dissipative (quantifying the energy dissipated by the discrete solution through its jumps), and show that energy arguments allow to control the trace of the Trefftz-DG Galerkin error on space-like mesh interfaces.

We call ``space-like interface'' a graph hypersurface 
$$\Sigma=\big\{(\bx,f_\Sigma(\bx)):\; \bx\in \Omega\big\}\subset\overline Q$$ 
where $f_\Sigma:\overline \Omega\to[0,T]$ is a Lipschitz-continuous function whose Lipschitz constant is smaller than $1/c$.
Each space-like mesh face in $\Fspa$ is subset of a space-like interface $\Sigma$.
The future-pointing unit normal vector is defined almost everywhere on $\Sigma$ and denoted by $(\bn^x_\Sigma,n^t_\Sigma)$.

For sufficiently smooth scalar and vector fields $(w,\btau)$, define their ``\emph{energy}''
on a space-like interface $\Sigma$ as
$$\calE(\Sigma;w,\btau):=\int_\Sigma \bigg(
w\,\btau\cdot \bn^x_\Sigma+\frac12\Big(c^{-2} w^2+|\btau|^2\Big)n^t_\Sigma
\bigg)\di S.
$$
The energy on constant-time, or ``flat'', space-like interfaces are denoted by $\calE(t;w,\btau):=\calE(\overline\Omega\times\{t\};w,\btau)$, for $0\le t\le T$.

Using $|\bn^x_F|\le\gamma c^{-1}n^t_F$ on $\Fspa$ and the weighted Cauchy--Schwarz inequality, we have lower and upper bounds for the energy in terms of $L^2$ norms:
if $\Sigma\subset\Fspa\cup\FO\cup\FT$ is a space-like interface composed by element faces, then for all $ (w,\btau)\in L^2(\Sigma)^{1+n}$
\begin{align}\label{eq:EnergyLowerBound}
&\frac12\int_\Sigma (1-\gamma) n^t_\Sigma\Big(c^{-2} w^2+|\btau|^2\Big)\di S
\le
\calE(\Sigma; w, \btau)
\le\frac12\int_\Sigma (1+\gamma) n^t_\Sigma\Big(c^{-2} w^2+|\btau|^2\Big)\di S.
\end{align}
In particular, for any non-zero $(w,\btau)\in L^2(\Sigma)^{1+n}$, from $\gamma<1$ we have $\calE(\Sigma;w,\btau)>0$.

For two space-like interfaces $\Sigma_1,\Sigma_2$ with $f_{\Sigma_1}\le f_{\Sigma_2}$ in $\overline \Omega$, we denote the volume between them and its lateral boundary as
\begin{align*}
Q_{\Sigma_1,\Sigma_2}&:=\{(\bx,t): \;\bx\in\Omega,\; f_{\Sigma_1}(\bx)<t<f_{\Sigma_2}(\bx)\},\\
\Gamma_{\Sigma_1,\Sigma_2}&:=\{(\bx,t):\; \bx\in\deO,\; f_{\Sigma_1}(\bx)\le t\le f_{\Sigma_2}(\bx)\}.
\end{align*}
For such interfaces, the following energy identity can be verified integrating by parts:
\begin{align}\label{eq:EnergyId}
\calE(\Sigma_2;w,\btau)=&
\calE(\Sigma_1;w,\btau)-\int_{\Gamma_{\Sigma_1,\Sigma_2}} w\,\btau\cdot\bn^x_\Omega\di S
\\&\nonumber
+\int_{Q_{\Sigma_1,\Sigma_2}}\bigg(
\Big(\nabla w+\der{\btau}t \Big) \cdot \btau
+\Big(\nabla\cdot\btau+c^{-2}\der{w}t\Big) w\bigg)
\di V.
\end{align}
If $(v,\bsigma)$ is the solution of IBVP~\eqref{eq:IBVP}, then we have
$$\calE(\Sigma_2;v,\bsigma)=
\calE(\Sigma_1;v,\bsigma)-\int_{\Gamma_{\Sigma_1,\Sigma_2}} v\,\bsigma\cdot\bn^x_\Omega\di S.
$$
If $g_D=g_N=g_R=0$ in their domains of definition, since $\vartheta\ge0$, we have 
$\calE(\Sigma_2;v,\bsigma)\le\calE(\Sigma_1;v,\bsigma)$, i.e.\ energy is dissipated.
If moreover $\FR=\emptyset$, then
$\calE(\Sigma_2;v,\bsigma)=\calE(\Sigma_1;v,\bsigma)$, 
i.e.\ energy is preserved.



\am{We verify that the method \eqref{eq:TDG} is \emph{dissipative}. Assume that $g_D=g_N=g_R=0$. Since 
\[
\calE(0;w,\btau)=\frac{1}{2}\N{c^{-1}w}_{L^2(\FO)}^2+\frac{1}{2}\N{\btau}_{L^2(\FO)}^2,
\quad
\calE(T;w,\btau)=\frac{1}{2}\N{c^{-1}w}_{L^2(\FT)}^2+\frac{1}{2}\N{\btau}_{L^2(\FT)}^2,
\]
using the definition~\eqref{eq:DGnorm} of the $DG$ norm, we can write
\begin{equation}\label{eq:new1}
\Tnorm{(\Vhp,\Shp)}\DG^2=\calE(0; \Vhp, \Shp)+\calE(T; \Vhp,
\Shp)+{\cal D},
\end{equation}
where
\begin{align*}
\cal D:=&
\frac12 \Norm{\Big(\frac{1-\gamma}{n^t_F}\Big)^{1/2}c^{-1}\jmp{\Vhp}_t}_{L^2(\Fspa)}^2
+\frac12\Norm{\Big(\frac{1-\gamma}{n^t_F}\Big)^{1/2}\jmp{\Shp}_t\rule{0pt}{3mm}}_{L^2(\Fspa)^n}^2 \\&
+\Norm{\alpha^{1/2}\jmp{\Vhp}_\bN}_{L^2(\Ftime)^n}^2
+\Norm{\beta^{1/2}\jmp{\Shp}_\bN}_{L^2(\Ftime)}^2\\
&+\Norm{\alpha^{1/2}\Vhp}_{L^2(\FD)}^2
+\Norm{\beta^{1/2} \Shp\cdot\bn_\Omega^x }_{L^2(\FN)}^2\\
&+\Norm{\Big(\frac{(1-\delta)\tht}c\Big)^{1/2} \Vhp}_{L^2(\FR)}^2
+\Norm{\Big(\frac{\delta c}\tht\Big)^{1/2} \Shp\cdot\bn_\Omega^x }_{L^2(\FR)}^2.
\end{align*}
Moreover, since $g_D=g_N=g_R=0$, then
\begin{equation}\label{eq:new2}
\ell_h(\Vhp,\Shp)=\int_{\FO}(c^{-2}v_0\Vhp+\bsigma_0\Shp)\,\di S\le
\calE(0;v_0,\bsigma_0)+\calE(0; \Vhp, \Shp).
\end{equation}
From the coercivity property proved at the beginning of Section~\ref{s:WellP}, we have the inequality
$\Tnorm{(\Vhp,\Shp)}\DG^2\le\ell(\Vhp,\Shp)$ and thus
\begin{align*}
\calE(T; \Vhp, \Shp)&\overset{\eqref{eq:new1}}{=}\Tnorm{(\Vhp,\Shp)}\DG^2-\calE(0; \Vhp,
  \Shp)-{\cal D}\\
&\;\le\; \ell_h(\Vhp,\Shp)-\calE(0; \Vhp,  \Shp)-{\cal D}
\qquad\quad\overset{\eqref{eq:new2}}{\le}\calE(0;v_0,\bsigma_0)-{\cal D}.
\end{align*}}
In other words, energy is dissipated by the \am{terms in $\cal D$, namely, the} discrete solution jumps across mesh interfaces, 
the mismatch with the (homogeneous) Dirichlet and Neumann boundary conditions, and 
both Dirichlet and Neumann traces on the Robin boundary.

From the definition of the $\Tnorm{\cdot}\DG$ norm and Theorem~\ref{th:QO}, we also have the following error bound on the energy of the Galerkin error at final time $T$:
\[
\calE(T; v-\Vhp,\bsigma-\Shp)\le \Tnorm{(v-\Vhp,\bsigma-\Shp)}^2\DG\le  (1+C_c)^2 \inf_{(w,\btau)\in\bVp\Th} \Tnorm{(v-w,\bsigma-\btau)}^2\DGp.
\]
Next proposition shows that this bound extends to any space-like interface
  $\Sigma\subset\Fh$.

\begin{prop}\label{prop:errorSigma}
Let $\Sigma\subset\Fh$ be a space-like interface. Then the following
error bound holds true:
\[
\calE(\Sigma; v-\Vhp^-,\bsigma-\Shp^-)
\le \frac52\N{(1-\gamma)^{-1}}_{L^\infty(\Sigma)}(1+C_c)^2
\inf_{(w,\btau)\in\bVp\Th} \Tnorm{(v-w,\bsigma-\btau)}^2\DGp,
\]
with $C_c$ as in \eqref{eq:Continuity}.
\end{prop}
\begin{proof}
Setting  
\[
Q_\Sigma:=\{(\bx,t): \;\bx\in\Omega,\; 0<t<f_{\Sigma}(\bx)\}, 
\]
the analytic solution of the IBVP analogue to \eqref{eq:IBVP} posed on $Q_\Sigma$ coincides with the restriction to $Q_\Sigma$ of the analytic solution of \eqref{eq:IBVP} (posed on the whole of $Q$).
The Trefftz-DG method on $Q_\Sigma$ gives a well-posed discrete problem and, since the numerical fluxes on $\Sigma$ are defined as $\hVhp=\Vhp^-$ and $\hShp=\Shp^-$, its solution coincides with the restriction to $Q_\Sigma$ of the solution of~\eqref{eq:TDG}.

Define $\bT_\Sigma\Th:=\{(w,\btau)\in\bT\Th,\supp(w,\btau)\subset\overline{Q_\Sigma}\}$ and denote by $\Tnorm{\cdot}\DGQ$ and $\Tnorm{\cdot}\DGQp$ the restriction to $\bT_\Sigma\Th$ of $\Tnorm{\cdot}\DG$ and $\Tnorm{\cdot}\DGp$, respectively.
In particular, for $(w,\btau)\in\bT_\Sigma\Th$, the terms on $\Sigma$ appearing in $\Tnorm{(w,\btau)}\DGQ^2$ and $\Tnorm{(w,\btau)}\DGQp^2$ are
\begin{align*}
&\frac12\Norm{(1-\gamma)^{1/2}(n_\Sigma^t)^{1/2}c^{-1}w^-}^2_{L^2(\Sigma)} 
+\frac12\Norm{(1-\gamma)^{1/2}(n_\Sigma^t)^{1/2}\btau^- \rule{0pt}{3mm}}^2_{L^2(\Sigma)^n}
\qquad\text{and}\\
&
2\Norm{\Big(\frac{n^t_F}{1-\gamma}\Big)^{1/2} c^{-1}w^-}_{L^2(\Sigma)}^2
+2\Norm{\Big(\frac{n^t_F}{1-\gamma}\Big)^{1/2}\btau^-}_{L^2(\Sigma)^n}^2,
\end{align*}
respectively.
Continuity and coercivity of the bilinear form $\calA$ hold for the subspace $\bT_\Sigma\Th$, thus, since the functions in $\bVp\Th$ are discontinuous,
\[
\Tnorm{(v-v_{hp},\bsigma-\bsigma_{hp})_{|_{Q_\Sigma}}}\DGQ
\le(1+C_c)
\inf_{(w,\btau)\in\bVp(\calT_h)
}\Tnorm{(v-w,\bsigma-\btau)_{|_{Q_\Sigma}}}\DGQp.
\]
This, together with the right bound of~\eqref{eq:EnergyLowerBound}, the definition
of $\Tnorm{\cdot}\DGQ$, and $1+\gamma< 2$, gives
\[
\calE(\Sigma; v-\Vhp^-,\bsigma-\Shp^-)
\le \N{\frac2{1-\gamma}}_{L^\infty(\Sigma)}(1+C_c)^2
\inf_{(w,\btau)\in\bVp(\calT_h)
}\Tnorm{(v-w,\bsigma-\btau)_{|_{Q_\Sigma}}}^2\DGQp.
\]
For any $(w,\btau)\in\bT\Th$, $\Tnorm{(w,\btau)}\DGp$ contains a term with the traces $(w^-,\btau^-)$ on $\Sigma$ and one with the jumps $(\jmp{w}_t,\jmp{\btau}_t)$, while $\Tnorm{(w,\btau)_{|_{Q_\Sigma}}}\DGQp$ contains two terms with the traces only.
Thus, using $\gamma<1$ and $\frac{1-\gamma}2+\frac2{1-\gamma}\le \frac54\frac2{1-\gamma}$, we have
$$
\Tnorm{(w,\btau)_{|_{Q_\Sigma}}}\DGQp^2
\le \frac54
\Tnorm{(w,\btau)}\DGp^2
\qquad \forall (w,\btau)\in\bT\Th,
$$
from which the proof of the statement is complete. 
\end{proof}
In other words, Proposition~\ref{prop:errorSigma} says that, on every space-like mesh face $F$ contained in a space-like interface $F\subset\Sigma\subset\Fspa$, the Trefftz-DG error is controlled in $L^2(F)^{1+n}$ norm by the approximation properties of the discrete space.

\subsection{Error bounds in mesh-independent norms}\label{s:MeshIndependent}

The bound~\eqref{eq:QuasiOpt} of Theorem~\ref{th:QO} allows to control the Trefftz-DG error on the mesh skeleton only.
We want to control the error in the space--time $L^2$-norm; we can do this only under some further assumptions (\S\ref{s:Stability}), while in the general case we can obtain a bound in a weaker mesh-independent norm (\S\ref{s:StabilityX}). 
To this purpose, we consider a general linear space $\bX\subset L^2(Q)^{1+n}$
with norm $\N{(w,\btau)}_\bX\ge C_X\N{(w,\btau)}_{L^2(Q)^{1+n}}$ possibly stronger than $L^2(Q)^{1+n}$, independent of the mesh $\calT_h$, and define the dual norm:
\begin{align}\label{eq:Xstar}
\N{(w,\btau)}_{\bX^*}:=
\sup_{\bzero\ne(\psi,\bPhi)\in \bX} \frac{\int_Q(w\psi+\btau\cdot\bPhi)\di x\di t}{\N{(\psi,\bPhi)}_\bX}.
\end{align}
The $\bX^*$ norm is not stronger than the $L^2(Q)^{1+n}$ norm:
$\N{(w,\btau)}_{\bX^*}\le C_X^{-1}\N{(w,\btau)}_{L^2(Q)^{1+n}}$ for all $(w,\btau)\in L^2(Q)^{1+n}$.


\subsubsection{Auxiliary problem and duality argument}
\label{s:Duality}
To control the $\bX^*$ norm of the Trefftz-DG error, we first consider the auxiliary inhomogeneous IBVP
\begin{align}\label{eq:zzIBVP}
\left\{\begin{aligned}
&\nabla z+ \partial\bzeta/\partial t 
= \bPhi &&\iin Q,\\
&\nabla\cdot\bzeta+c^{-2} \,\partial z/\partial t 
= \psi &&\iin Q,\\
&z(\cdot,0)=0, \quad \bzeta(\cdot,0)=\bzero\qquad &&\oon \Omega,\\
&z=0 &&\oon \GD\times I,\\
&\bzeta\cdot\bn_\Omega^x=0 &&\oon \GN\times I,\\
&\frac\tht c z-\bzeta\cdot\bn_\Omega^x=0 &&\oon \GR\times I,
\end{aligned}\right.\end{align}
with data $\psi\in L^2(Q),\bPhi\in L^2(Q)^n$.
We also define a mesh-skeleton seminorm for a pair $(z,\bzeta)\in L^2(Q)^{1+n}$ with sufficiently regular traces:
\begin{align}
\abs{(z,\bzeta)}_{\Fh}:=&
\Bigg(2\Norm{\Big(\frac{(1+\gamma^2)n^t_F}{1-\gamma}\Big)^{1/2}c^{-1}z}^2_{L^2(\Fspa\cup\FT)}
+2\Norm{\Big(\frac{(1+\gamma^2)n^t_F}{1-\gamma}\Big)^{1/2}\bzeta}^2_{L^2(\Fspa\cup\FT)^n}
\nonumber\\&
+\Norm{\alpha^{-1/2}\bzeta\cdot\bn^x_F}^2_{L^2(\Ftime\cup\FD)}
+\Norm{\beta^{-1/2}z}^2_{L^2(\Ftime\cup\FN)}
\nonumber\\&
+\Norm{\Big(\frac c{(1-\delta)\tht}\Big)^{1/2}\bzeta\cdot\bn^x_\Omega}^2_{L^2(\FR)}
+\Norm{\Big(\frac \tht{\delta c}\Big)^{1/2}z}^2_{L^2(\FR)}
\Bigg)^{1/2}.
\label{eq:FhNorm}
\end{align}
In the next proposition, we show that the $\Tnorm{\cdot}\DG$ norm of a
Trefftz field controls a mesh-independent norm of the same field. 
In particular, this holds for the Trefftz-DG error.
The key assumption is the following one. 

\begin{assum}\label{ass:DualStability}
There exists a positive constant $\Cstab$,
depending on the domain $Q$, on the mesh $\calT_h$ (thus on $\gamma$),
and on the parameters $c,\tht,\alpha,\beta,\delta$, such that
\begin{align}\nonumber
\forall(\psi,\bPhi)\in \bX,&
\text{ the solution $(z,\bzeta)$ of \eqref{eq:zzIBVP} satisfies the stability bound:}\\
&\abs{(z,\bzeta)}_{\Fh}
\le\Cstab \Norm{(\psi,\bPhi)}_\bX.
\label{eq:DualStability}
\end{align}
\end{assum}

\am{The quantity $\Norm{(\psi,\bPhi)}_\bX$ on the right-hand side of \eqref{eq:DualStability} is independent of $\calT_h$, while the left-hand side is an integral over the mesh skeleton (see \eqref{eq:FhNorm}): this implies that the value of $\Cstab$ necessarily blows up when the mesh $\calT_h$ is uniformly refined.}

Conditions under which Assumption~\ref{ass:DualStability}
holds are given in \S\ref{s:Stability}, for $\N{\cdot}_\bX$ equal to the   $L^2(Q)^{1+n}$ norm, and in \S\ref{s:StabilityX}, for $\N{\cdot}_\bX$ stronger than the $L^2(Q)^{1+n}$ norm.

\begin{proposition}\label{prop:Duality}
If Assumption~\ref{ass:DualStability} is satisfied, then for all Trefftz fields $(w,\btau)\in\bT\Th$
\begin{align}
\N{(w,\btau)}_{\bX^*}
\le \Cstab \Tnorm{(w,\btau)}\DG.
\label{eq:zzDualityBound}
\end{align}
\end{proposition}
\begin{proof}
We first prove the vanishing of certain jumps across mesh faces for the solution $(z,\bzeta)$ of the inhomogeneous auxiliary problem \eqref{eq:zzIBVP}:
$\jmp{z}_t$ and $\jmp{\bzeta}_t$ on $\Fspa$ and of $\jmp{z}_\bN$ and $\jmp{\bzeta}_\bN$ on $\Ftime$.
%
%
%
Given a hyperplane $\Pi=\{\bn^x_\Pi\cdot \bx+n^t_\Pi\,t=C_\Pi\}$, 
denote the scalar jump of a
function defined in $Q\setminus\Pi$ and admitting traces on $\Pi$ from both sides
as $\jmp{f}_\Pi:=f_{|\{\bn^x_\Pi\cdot \bx+n^t_\Pi\,t>C_\Pi\}}-f_{|\{\bn^x_\Pi\cdot \bx+n^t_\Pi\,t<C_\Pi\}}$.
From \eqref{eq:zzIBVP}, the fields $(\bzeta, c^{-2}z)$ and $(z\be_j,\bzeta_j)$, $1\le j\le n$ ($\be_j$ denoting the standard basis elements of $\IR^n$), are in $H(\mathrm{div}_{x,t};Q)$, thus their \textit{normal} jumps vanish\footnote{\am{Recall
that if $\bF\in H(\dive;D)$, $D=D_1\cup D_2\cup\Sigma\subset\IR^d$, $d\in\IN$, where $D,D_1,D_2$ are Lipschitz domains with outward-pointing unit vectors $\bn, \bn_1,\bn_2$, respectively and $\Sigma$ is a Lipschitz hypersurface separating $D_1$ and $D_2$, then, by the divergence theorem and the well-definiteness of the normal traces in $H(\dive;D)$ \cite[eq.~(2.6)]{ABD98},
\begin{align*}
\int_\Sigma (\bF_{|_{D_1}}\cdot\bn_1+ \bF_{|_{D_2}}\cdot\bn_2)\di S
&= \int_{\partial D_1}\bF_{|_{D_1}}\cdot\bn_1\di S
+ \int_{\partial D_2}\bF_{|_{D_2}}\cdot\bn_2\di S
- \int_{\partial D}\bF\cdot\bn\di S\\
&= \int_{D_1} \nabla\cdot\bF_{|_{D_1}}\di V
+ \int_{D_2} \nabla\cdot\bF_{|_{D_2}}\di V
- \int_{D} \nabla\cdot\bF\di V
=0.
\end{align*}
}}
across any space--time Lipschitz interface in $Q$ and in particular across $\Pi$:
$$\jmp{\bzeta\cdot\bn^x_\Pi + c^{-2}z n^t_\Pi}_\Pi=\jmp{z(\bn^x_\Pi)_j + \bzeta_j n^t_\Pi}_\Pi=0
\qquad 1\le j\le n.$$
Thus, on time-like faces $n^t_\Pi=0$, the jump of $z$ and the normal jump of $\bzeta$ vanish.
On constant-time faces ($\bn^x_\Pi=\bzero, n^t_\Pi=\pm1$) all jumps vanish \am{(recall that $c$ may jump only across time-like faces)}.
On other hyperplanes, $n^t_\Pi\ne0$ and $|\bn^x_\Pi|\ne 0$, thus
$$(-c^2/n^t_\Pi)\jmp{\bzeta\cdot\bn^x_\Pi }_\Pi= \jmp{z}_\Pi
=(-n^t_\Pi/|\bn^x_\Pi|^2)\jmp{\bzeta\cdot\bn^x_\Pi }_\Pi.$$
If $n^t_\Pi/|\bn^x_\Pi|\ne c$ then we have immediately $\jmp{z}_\Pi=\jmp{\bzeta\cdot\bn^x_\Pi }_\Pi=0$  and from above, $\jmp{\bzeta_j}_\Pi=0$ for all $1\le j\le n$. 
Assumption \eqref{eq:HorVerFaces} guarantees that $n^t_\Pi/|\bn^x_\Pi|> c$ on $\Fspa$, so we conclude that all jumps vanish.
(We have simply shown that the discontinuities of solutions of the first-order wave equations with source term in $L^2(Q)^{n+1}$ propagate along characteristics.)

Since we want to control the $\bX^*$ norm defined in \eqref{eq:Xstar},
we now take the scalar product of the Trefftz field $(w,\btau)$ with the source terms $(\psi,\bPhi)$ of problem \eqref{eq:zzIBVP} and integrate by parts in each element:
\begin{align*}
\int_Q(w\psi+\btau\cdot \bPhi)\di V
\overset{\eqref{eq:zzIBVP}}=\!&
\sum_{K\in\calT_h}\int_K \bigg(w\nabla\cdot\bzeta+c^{-2}{w}\der zt
+\btau\cdot\nabla z+\btau\cdot \der\bzeta t\bigg)\di V\\
=&\sum_{K\in\calT_h}\int_\deK \Big(
w\bzeta\cdot\bn^x_K+\btau\cdot\bn^x_K z+c^{-2}{wz n^t_K}+\btau\cdot \bzeta n^t_K\Big)\di S
\\
=&\int_\Fspa\underbrace{\jmp{w\bzeta+\btau z}_\bN+\jmp{{c^{-2}}wz +\btau\cdot \bzeta}_t}_{
\le c^{-1}|\jmp{w}_t|(\gamma |\bzeta|+c^{-1}|z|)+|\jmp{\btau}_t|(\gamma c^{-1}|z|+|\bzeta|)}
\di S
\\&
+\int_{\FT}\Big(c^{-2}{wz} +\btau\cdot \bzeta\Big)\di S
-\int_{\FO}\Big(c^{-2}w\underbrace{z}_{=0} 
+\btau\cdot \underbrace{\bzeta}_{=\bzero}\Big)\di S
\\&
+\int_{\Ftime}\underbrace{\jmp{w\bzeta+\btau z}_\bN}_{=\jmp{w}_\bN\cdot\bzeta+\jmp{\btau}_\bN z}\di S
\\&
+\int_{\FD\cup\FN\cup\FR}\big(w\underbrace{\bzeta\cdot\bn^x_\Omega}_{=0 \oon\FN}
+\btau\cdot\bn^x_\Omega \underbrace{z}_{=0\oon \FD}\big)\di S
\\
\le&
\Tnorm{(w,\btau)}\DG \cdot\bigg(
2\Norm{\Big(\frac{(1+\gamma^2)n^t_F}{1-\gamma}\Big)^{1/2}c^{-1}z}^2_{L^2(\Fspa)}
\\&
+2\Norm{\Big(\frac{(1+\gamma^2)n^t_F}{1-\gamma}\Big)^{1/2}\bzeta}^2_{L^2(\Fspa)^n}
+2\Norm{c^{-1}z}^2_{L^2(\FT)}
+2\Norm{\bzeta}^2_{L^2(\FT)^n}
\\&
+\Norm{\alpha^{-1/2}\bzeta\cdot\bn^x_F}^2_{L^2(\Ftime\cup\FD)}
+\Norm{\beta^{-1/2}z}^2_{L^2(\Ftime\cup\FN)}
\\&
+\Norm{\Big(\frac c{(1-\delta)\tht}\Big)^{1/2}\bzeta\cdot\bn^x_\Omega}^2_{L^2(\FR)}
+\Norm{\Big(\frac \tht{\delta c}\Big)^{1/2}z}^2_{L^2(\FR)}
\bigg)^{1/2}
\\
\overset{\eqref{eq:DualStability}}\le& \Cstab\Tnorm{(w,\btau)}\DG
\Norm{(\psi,\bPhi)}_\bX.
\end{align*}
Inserting this bound in the
definition \eqref{eq:Xstar} of the $\bX^*$ norm
of $(w,\btau)$, we obtain assertion~\eqref{eq:zzDualityBound}.
\end{proof}

If Assumption~\ref{ass:DualStability} is verified with $\bX=L^2(Q)^{1+n}$,
from Proposition \ref{prop:Duality} and the quasi-optimality Theorem \ref{th:QO}, it follows that the $L^2(Q)$ norm of the Trefftz-DG error is controlled by the $\Tnorm{\cdot}\DGp$ norm of the best-approximation error.
In \cite{SpaceTimeTDG}, bound \eqref{eq:DualStability} with $\bX=L^2(Q)^{1+n}$ was proven in one space dimension on meshes made of rectangular elements aligned to the space--time axes and $\Cstab$ was computed explicitly.
Two proofs were given.
One of them (Appendix~A of \cite{SpaceTimeTDG}) relies on the use of the exact value of $(z,\bzeta)$ in $Q$ computed with Duhamel's principle, 
and cannot be easily extended to general domains in higher space dimensions, as it require a suitable periodic extension of the IBVP \eqref{eq:zzIBVP} to $\IR^n\times(0,T)$.
The second proof (Lemma~4.9 of \cite{SpaceTimeTDG}) uses an energy argument to control the traces on space-like faces in \eqref{eq:DualStability} and an integration by parts trick to bound the traces on time-like faces.
In higher space dimensions, the energy argument carries over, while the traces on time-like faces are harder to control.
In \S\ref{s:Stability} we follow this idea and prove Assumption~\ref{ass:DualStability} in any dimension, under two additional assumptions, namely $\Ftime=\emptyset$ and $\GD=\GN=\emptyset$,
to get around the need to control traces on time-like faces.
We will make use of the energy identities and bounds discussed in Section~\ref{s:Energy}.
The traces of $z$ and $\bzeta$ on time-like faces are controlled by a stronger norm of $(\psi,\bPhi)$ in Proposition~\ref{prop:TimeFaceBounds}.


\subsubsection{Stability of the auxiliary problem: case without time-like faces}\label{s:Stability}

For meshes with no time-like faces, Assumption~\ref{ass:DualStability} is satisfied with $\bX=L^2(Q)^{1+n}$.
This is a consequence of following stability bound.

\begin{proposition}\label{prop:DualStability}
For all $(\psi,\bPhi)\in L^2(Q)^{1+n}$, the solution $(z,\bzeta)$ of the IBVP \eqref{eq:zzIBVP} satisfies the bound
\begin{align}
&\Bigg(2\Norm{\Big(\frac{(1+\gamma^2)n^t_F}{1-\gamma}\Big)^{1/2}c^{-1}z}^2_{L^2(\Fspa\cup\FT)}
+2\Norm{\Big(\frac{(1+\gamma^2)n^t_F}{1-\gamma}\Big)^{1/2}\bzeta}^2_{L^2(\Fspa\cup\FT)^n}
\nonumber\\&\nonumber
+\Norm{\Big(\frac c{(1-\delta)\tht}\Big)^{1/2}\bzeta\cdot\bn^x_\Omega}^2_{L^2(\FR)}
+\Norm{\Big(\frac \tht{\delta c}\Big)^{1/2}z}^2_{L^2(\FR)}
\Bigg)
\\&
\le\tCstab^2\Big(\N{c\psi}^2_{L^2(Q)}+\N{\bPhi}^2_{L^2(Q)^n}\Big)
\label{eq:NoTimeLikeStability}
\end{align}
with constant
\begin{align}\label{eq:CstabNoTimeLike}
\tCstab^2=2T\bigg(
N\Norm{\frac{4(1+\gamma^2)}{(1-\gamma)^2}}_{L^\infty(\Fspa)}
+\Norm{\frac1{\delta(1-\delta)}}_{L^\infty(\FR)}
\bigg),
\end{align}
where $N$ is the minimal number of space-like interfaces $\Sigma_1,\ldots,\Sigma_N$ such that 
$\Fspa\subset\bigcup_{1\le j\le N-1}\Sigma_j$ and $0\le f_{\Sigma_1}\le\cdots\le f_{\Sigma_{N-1}}\le f_{\Sigma_N}=T$.
\end{proposition}
\am{Due to the presence of the coefficient $N$ in \eqref{eq:CstabNoTimeLike}, the value of $\tCstab$ increases when the mesh is refined in time; if the refinement is uniform we have $\tCstab\approx h_t^{-1/2}$, $h_t$ being the time-step.}
\begin{proof}
Applying the energy identity \eqref{eq:EnergyId} to the solution $(z,\bzeta)$ of the IBVP \eqref{eq:zzIBVP}, we have that for any two space-like interfaces $\Sigma_1,\Sigma_2$ with $f_{\Sigma_1}\le f_{\Sigma_2}$,
\begin{align}\label{eq:zzEnergyEvolution}
\calE(\Sigma_2; z,\bzeta)\le
\calE(\Sigma_1; z,\bzeta)
+\int_{Q_{\Sigma_1,\Sigma_2}}\big(\bPhi \cdot \bzeta+\psi z\big)\di V
\end{align}
(equality holds if $\Gamma_{\Sigma_1,\Sigma_2}\cap\FR$ has vanishing $n$-dimensional measure).
This implies a bound in space--time $L^2$ norm:
\begin{align*}
\Norm{c^{-1}z}^2_{L^2(Q)}\!+\Norm{\bzeta}^2_{L^2(Q)^n}
&=2\int_0^T \calE(t;z,\bzeta)\di t
\\&
\overset{\eqref{eq:zzEnergyEvolution}}\le 2\int_0^T 
\Big(\underbrace{\calE(0;z,\bzeta)}_{=0} 
+\int_\Omega\int_0^t(\bPhi\cdot\bzeta+\psi z)\di s\di \bx\Big)\di t\\
&\le 2T \Big(\Norm{c^{-1}z}^2_{L^2(Q)}+\Norm{\bzeta}^2_{L^2(Q)^n}\Big)^{1/2}
\Big(\Norm{c\psi}^2_{L^2(Q)}+\Norm{\bPhi}^2_{L^2(Q)^n}\Big)^{1/2},
\end{align*}
from which
\begin{align}\label{eq:zzL2QBound}
\Norm{c^{-1}z}^2_{L^2(Q)}+\Norm{\bzeta}^2_{L^2(Q)^n}
\le 4T^2 \big(\Norm{c\psi}^2_{L^2(Q)}+\Norm{\bPhi}^2_{L^2(Q)^n}\big).
\end{align}
For every space-like mesh interface $\Sigma\subset\Fh$ we control the corresponding term in \eqref{eq:DualStability} with the energy term:
\begin{align}
\label{eq:zzCgamma}
&2\Norm{\Big(\frac{(1+\gamma^2)n^t_\Sigma}{1-\gamma}\Big)^{1/2}c^{-1}z}^2_{L^2(\Sigma)}
+2\Norm{\Big(\frac{(1+\gamma^2)n^t_\Sigma}{1-\gamma}\Big)^{1/2}\bzeta}^2_{L^2(\Sigma)^n}
\\&\qquad
\le \underbrace{\Norm{\frac{4(1+\gamma^2)}{(1-\gamma)^2}}_{L^\infty(\Fspa)}}_{=:C_\gamma}
\frac12\int_\Sigma (1-\gamma) n^t_\Sigma\Big(\frac{z^2}{c^2}+|\bzeta|^2\Big)\di S
\overset{\eqref{eq:EnergyLowerBound}}\le C_\gamma\;\calE(\Sigma;z,\bzeta).
\nonumber
\end{align}
We partition the faces in $\Fspa$ into $(N-1)$ interfaces $\Sigma_j$ with $\Fspa\subset\bigcup_{1\le j\le N-1}\Sigma_j$  such that 
$0\le f_{\Sigma_1}\le\cdots\le f_{\Sigma_{N-1}}\le 
T$ and denote $\Sigma_N=\Omega\times\{T\}$.
We now control all terms on the space-like faces:
\begin{align*}
&2\Norm{\Big(\frac{(1+\gamma^2)n^t_F}{1-\gamma}\Big)^{1/2}c^{-1}z}^2_{L^2(\Fspa\cup\FT)}
+2\Norm{\Big(\frac{(1+\gamma^2)n^t_F}{1-\gamma}\Big)^{1/2}\bzeta}^2_{L^2(\Fspa\cup\FT)^n}
\\&
\overset{\eqref{eq:zzCgamma}}\le C_\gamma \sum_{j=1}^N \calE(\Sigma_j;z,\bzeta)
\\&
\overset{\eqref{eq:zzEnergyEvolution}}\le C_\gamma \sum_{j=1}^N \int_{Q_{\Omega\times\{0\},\Sigma_j}}\big(\bPhi \cdot \bzeta+\psi z\big)\di V
\\&
\overset{\eqref{eq:zzL2QBound}}\le
2C_\gamma N T\big(\Norm{c\psi}^2_{L^2(Q)}+\Norm{\bPhi}^2_{L^2(Q)^n}\big).
\end{align*}
We are now left with the terms on $\FR$:
using the Robin boundary condition $\frac\tht c z=\bzeta\cdot\bn_\Omega^x$,
the energy identity \eqref{eq:EnergyId} and the $L^2(Q)$ stability bound \eqref{eq:zzL2QBound}, we have
\begin{align*}
&\hspace{-10mm}\Norm{\Big(\frac c{(1-\delta)\tht}\Big)^{1/2}\bzeta\cdot\bn^x_\Omega}^2_{L^2(\FR)}
+\Norm{\Big(\frac \tht{\delta c}\Big)^{1/2}z}^2_{L^2(\FR)}
\\&
\le\underbrace{
\Norm{\frac1{\delta(1-\delta)}}_{L^\infty(\FR)}}_{=:C_\delta}
\int_\FR\Big(
\delta\frac c\tht(\bzeta\cdot\bn^x_\Omega)^2+ (1-\delta)\frac\tht cz^2\Big)\di S
\\&
\overset{\tht z=c\bzeta\cdot\bn_\Omega^x}=
C_\delta
\int_\FR z\:\bzeta\cdot\bn^x_\Omega\di S
\\&
\overset{\eqref{eq:EnergyId}}= C_\delta\bigg(
\underbrace{\calE(0;z,\bzeta)}_{=0}-\underbrace{\calE(T;z,\bzeta)}_{\ge0}
+\int_Q\big(\bPhi\cdot\bzeta+\psi\, z\big)\di V
\bigg)
\\&
\overset{\eqref{eq:zzL2QBound}}\le C_\delta 2T \big(\Norm{c\psi}^2_{L^2(Q)}+\Norm{\bPhi}^2_{L^2(Q)^n}\big).
\end{align*}
Combining this inequality with the previous one we obtain the assertion with $\tCstab^2=2T(C_\gamma N+C_\delta)$.
\end{proof}
Note that in Proposition~\ref{prop:DualStability} we do not require $\Sigma_j\subset\Fspa$.

We are now ready to prove error bounds in $L^2(Q)$ norm 
through bound \eqref{eq:DualStability}, under some further assumption.
In particular, under assumption (ii) below, the wave speed $c$ must be constant throughout $Q$.
A mesh satisfying this assumption is depicted in Figure \ref{fig:Mesh}.

\begin{cor}\label{cor:NoTimeLike}
Assume that
\begin{itemize}
\item[(i)]
$\GD=\GN=\emptyset$, i.e.\ only Robin boundary conditions are allowed ($\deO=\GR$); and
\item[(ii)]
$\Ftime=\emptyset$, i.e.\ no time-like mesh interfaces are allowed.
\end{itemize}
Then the $L^2(Q)^{1+n}$ norm of the Trefftz-DG error is controlled:
\begin{align}\label{eq:L2Error}
\Big(\Norm{c^{-1}(v-\Vhp)}_{L^2(Q)}^2&+\Norm{\bsigma-\Shp}_{L^2(Q)^n}^2\Big)^{1/2}\\
&\le  \Cstab(1+C_c) \inf_{(w,\btau)\in\bVp\Th} \Tnorm{(v-w,\bsigma-\btau)}\DGp,
\nonumber
\end{align}
where $\Cstab$ coincides with $\tCstab$ defined in \eqref{eq:CstabNoTimeLike}, and $C_c$ is as in \eqref{eq:Continuity}.
\end{cor}
\begin{proof}
Under assumptions (i)--(ii), the $\abs{\cdot}_{\Fh}$ seminorm in \eqref{eq:FhNorm} reduces to the left-hand side of \eqref{eq:NoTimeLikeStability}, thus Proposition~\ref{prop:DualStability} gives the stability bound \eqref{eq:DualStability} with 
$\Norm{(\psi,\bPhi)}_{\bX}^2=\N{c\psi}^2_{L^2(Q)}+\N{\bPhi}^2_{L^2(Q)^n}$
and $\Cstab=\tCstab$.
From the duality argument of Proposition~\ref{prop:Duality} and the quasi-optimality of Theorem~\ref{th:QO}, we have
\begin{align*}
\Big(\Norm{c^{-1}(v-\Vhp)}_{L^2(Q)}^2&+\Norm{\bsigma-\Shp}_{L^2(Q)^n}^2\Big)^{1/2}\\
&\overset{\eqref{eq:Xstar}}=\N{(v-\Vhp,\bsigma-\Shp)}_{\bX^*}
\\
&\overset{\eqref{eq:zzDualityBound}}\le \Cstab\Tnorm{(v-\Vhp,\bsigma-\Shp)}\DG
\\
&\overset{\eqref{eq:QuasiOpt}}\le \Cstab (1+C_c) \inf_{(w,\btau)\in\bVp\Th} \Tnorm{(v-w,\bsigma-\btau)}\DGp.
\end{align*}
\end{proof}

Assumption (ii) in Corollary \ref{cor:NoTimeLike} requires that all the internal mesh faces are space-like; Figure \ref{fig:Mesh} shows a mesh of this kind.
The meshes that satisfy this condition allow the Trefftz-DG method to be treated as a ``semi-explicit'' scheme as in \cite{FalkRichter1999,MoRi05,GSW16}:
if the elements are suitably designed and ordered, the discrete solution can be computed sequentially solving a local problem for each element.
This also allows a high degree of parallelism.
If the ``\emph{tent-pitching}'' algorithm of \cite{EGSU05,UnSh02,GSW16}
 is used to construct the mesh and the ``macro elements'' of \cite{MoRi05} are taken as elements, then the mesh obtained satisfies the assumptions of Proposition \ref{prop:DualStability}. 
The fact that the elements obtained in this way do not have simple shapes such as $(n+1)$-simplices is not a computational difficulty: all integrals in the Trefftz-DG formulation \eqref{eq:TDG} are defined on mesh faces, which are $n$-simplices, thus no quadrature on complicated shapes needs to be performed.
This is due to the Trefftz property, so this advantage is not available to discretisations employing standard (non-Trefftz) local spaces.

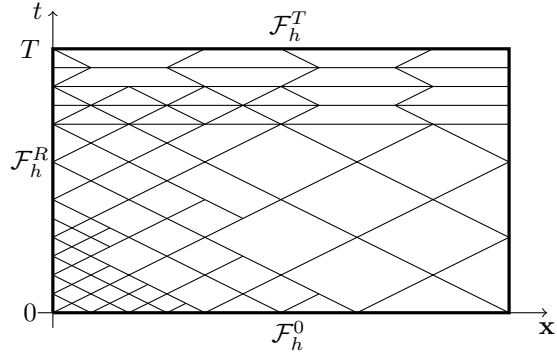
\begin{figure}[htb]
\centering
\begin{tikzpicture}[scale=1]
\draw [->] (0,-0.2)--(0,4);\draw(-0.2,4)node{$t$};
\draw [->] (-0.2,0) -- (6.5,0);\draw(6.5,-0.2)node{$\bx$}; 
\draw [very thick] (0,0)--(6,0)--(6,3.5)--(0,3.5)--(0,0);
\draw (0,0)--(5,2.5)--(6,2)--(2,0)--(0,1)--(3,2.5)--(6,1)--(4,0)--(0,2)--(1,2.5)--(6,0);
\draw (3.5,0.25)--(3,0)--(0,1.5)--(1.5,2.25)--(2,2.5);;
\draw (2.5,0.75)--(1,0)--(0,0.5)--(2,1.5); 
\draw (0,2.5)--(2.5,1.25);
\draw (1.75,0.125)--(1.5,0)--(0,0.75)--(0.75,1.125);
\draw (1.25,0.375)--(0.5,0)--(0,0.25)--(0.75,0.625);
\draw (0.75,0.875)--(0,1.25);
\draw(3.1,-0.3)node{$\FO$}; 
\draw(3.1,3.8)node{$\FT$}; 
\draw(-0.3,2)node{$\FR$}; 
\draw(-0.3,0)node{$0$}; 
\draw(-0.3,3.5)node{$T$}; 
\draw(0,2.5)--(6,2.5);
\draw(0,2.75)--(6,2.75);
\draw(0,3)--(6,3);
\draw(0,3.25)--(6,3.25);
\draw(2,3)--(1.5,3.25)--(2,3.5);
\draw(0,3)--(0.5,3.25)--(0,3.5);
\draw(3,2.5)--(3.5,2.75)--(3,3)--(3.5,3.25)--(3,3.5);
\draw(5,2.5)--(4.5,2.75)--(5,3)--(4.5,3.25)--(5,3.5);
\draw(0,2.5)--(0.5,2.75)--(1,2.5)--(1.5,2.75)--(2,2.5)--(2.5,2.75)--(3,2.5);
\draw(0,3)--(0.5,2.75)--(1,3)--(1.5,2.75)--(2,3)--(2.5,2.75)--(3,3);
\end{tikzpicture}
\caption{An example of a mesh in one space dimension ($n=1$) satisfying the assumptions of Proposition \ref{prop:DualStability}.
All internal mesh faces are space-like.
Not all mesh elements are 2-simplices (triangles), but the faces are 1-simplices (segments), so all integrals in \eqref{eq:TDG} are easy to compute.
For this mesh, the parameter $N$ in the proof of Proposition \ref{prop:DualStability} is equal to 
24. For images of tent-pitched meshes satisfying the same assumptions in higher dimensions, see e.g.\ \cite{GSW16,UnSh02,EGSU05}.
\label{fig:Mesh}
}
\end{figure}



\subsubsection{Stability of the auxiliary problem: case with time-like faces}\label{s:StabilityX}

Corollary~\ref{cor:NoTimeLike} allows to control the $L^2(Q)^{1+n}$
norm of the Trefftz-DG error only if the mesh does not contain
time-like faces and only Robin boundary conditions are used.
This is because on time-like faces the $L^2$ norm of the trace of $z$ and of the normal
trace of $\bzeta$, solution to the auxiliary  problem~\eqref{eq:zzIBVP},  seem not to be bounded by the $L^2(Q)^{1+n}$ norm of the sources $(\psi,\bPhi)$; compare the left-hand side of \eqref{eq:NoTimeLikeStability} and the seminorm \eqref{eq:FhNorm} we would like to bound.
The simple integration-by-parts trick used in one space dimension (see the final part of the proof of \cite[Lemma~4.9]{SpaceTimeTDG}) can not immediately be applied in higher dimensions; this is related to the fact that $\Hdiv=H^1\OO$ in 1D, so this space admits $L^2$ traces.
To prove analogues bounds in the presence of time-like faces, we need to exploit the regularity of the solutions $(z,\bzeta)$ of the inhomogeneous wave equations.
This requires to measure $\psi$ and $\bPhi$ in a norm stronger than $L^2(Q)$, which is the $\bX$ norm, leading to bounds on the error in the weaker norm $\bX^*$.

In this section, we restrict ourselves to the following situation:
\begin{itemize}
\item the space dimension is $n=3$,
\item only Dirichlet boundary conditions are present, i.e.\ $\partial\Omega=\GD$ and $\GN=\GR=\emptyset$,
\item the mesh elements are Cartesian products of polyhedra in space and intervals in time.
\end{itemize}
although we expect that the argument can be extended to much more general IBVPs and discretisations.
We use standard Bochner space notation for spaces and norms as in \cite[Sect.~5.9.2]{EVA02}. 
This is the only part of the paper where we do not make the bounding constants explicit.

\begin{prop}\label{prop:TimeFaceBounds}
Assume $\deO=\GD$ (so $\FN=\FR=\emptyset$) and $n=3$.
Define $\bX$ to be the closure of $C^\infty_0(Q)^{1+n}$ with respect to the norm
\begin{align}
\nonumber
\N{(\psi,\bPhi)}_\bX^2:=
\N{\psi}_\LtQ^2
&+\N{\bPhi}_{\LtQ^3}^2
+\N{\der{\psi}{t}-\nabla\cdot\bPhi}_{\LtQ}^2
\\&
+\N{(\bx,t)\mapsto\int_0^t\nabla\times\bPhi(\bx,s)\di s}_{L^2(Q)^3}^2.
\label{eq:XNorm3D}
\end{align}
Assume that all mesh elements $K\in\calT_h$ are space--time Cartesian products in the form
$K=K_0\times(t_K^-,t_K^+)$.  
Assume that there exists $\rho>0$ such that, for all elements $K\in\calT_h$, the space projection $K_0$ is star-shaped\footnote{We recall that a set $A\subset\IR^N$ is called star-shaped with respect to a subset  $B\subset A$ if for all $\ba\in A$ and $\bb\in B$ the line segment with endpoints $\ba$ and $\bb$ is contained in $A$. In particular, a convex set is star-shaped with respect to any of its subsets.}
with respect to a $n$-dimensional ball with radius $\rho\diam(K_0)$.
Define a meshsize and two wavespeed functions on $\Ftime\cup\FD$:
\begin{align*}
\hh,\ttc_-,\ttc_+\in L^\infty(\Ftime\cup\FD),\qquad
\hh(\bx,t):=&\min\{\diam(K_0): K\in\calT_h, (\bx,t)\in \am{\overline{K}}\},\\
\ttc_-(\bx,t):=&\min\{c_{|_K}: K\in\calT_h, (\bx,t)\in \am{\overline{K}}\},\\
\ttc_+(\bx,t):=&\max\{c_{|_K}: K\in\calT_h, (\bx,t)\in \am{\overline{K}}\},
\\
h_\calT:=&\N{\hh}_{L^\infty(\Ftime\cup\FD)}.
\end{align*}
Define two arbitrary positive functions 
$$\tta\in L^\infty(\Ftime\cup\FD),\quad\ttb\in L^\infty(\Ftime)
\quad
0<\tta_*\le\tta\le\tta^*,\quad
0<\ttb_*\le\ttb\le\ttb^*,
$$
for some constants $\tta_*,\tta^*,\ttb_*,\ttb^*$,
and fix the coefficients of the Trefftz-DG numerical fluxes as
$$
\alpha=\frac{h_\calT\tta}{\ttc_-\hh}, 
\qquad
\beta= \frac{\ttc_+ h_\calT\ttb}{\hh}.
$$

Then, for all $(\psi,\bPhi)\in\bX$, the solution $(z,\bzeta)$ of~\eqref{eq:zzIBVP} satisfies
$$
\Norm{\alpha^{-1/2}\bzeta\cdot\bn^x_F}^2_{L^2(\Ftime\cup\FD)}
+\Norm{\beta^{-1/2}z}^2_{L^2(\Ftime)}
\le C_{S,\rho} (\tta_*^{-1}+\ttb_*^{-1}) h_\calT^{-1} \N{(\psi,\bPhi)}_\bX^2,
$$
where $C_{S,\rho}>0$ only depends on $\Omega,T,c$ and $\rho$.
\end{prop}
\begin{proof}
By the density of $C^\infty_0(Q)^{1+n}$ in $\bX$, it is enough to consider the case
$(\psi,\bPhi)\in C^\infty_0(Q)^{1+n}$, i.e.\ $\psi$ and $\bPhi$ are smooth fields supported in the interior of $Q$.

We make use of the classical regularity result for the second-order wave equation on Lipschitz domains of \cite[Chapter 3, Theorems~8.1 and 8.2]{LiMaI}, applied with $H=L^2\OO$ and $V=H^1_0\OO$.
If $f\in L^2(Q)$, then the solution $u$ to the inhomogeneous IBVP
\begin{align}\label{eq:IBVP_uf}
\begin{cases}
\displaystyle
-\Delta u + c^{-2}\der{{}^2}{t^2}u = f&\iin Q,\\
\displaystyle
u(\bx,0)=0,\quad
\der ut(\bx,0)=0&\oon\Omega,\\
u=0 &\oon \deO\times (0,T)
\end{cases}
\end{align}
belongs to $C^0([0,T];H^1\OO)\cap C^1([0,T];L^2\OO)$ and satisfies the stability bound
\begin{align*}
\N{u}_{L^2(Q)} + \N{\nabla u}_{L^2(Q)^3} + \N{\der ut}_{L^2(Q)}
\le  C_S \N{f}_{L^2(0,T;L^2\OO)}.
\end{align*}
Here and in the rest of this proof, we denote by $C_S$ a positive constant that depends
only on $\Omega$, $T$, and $c$ and whose value may change at each occurrence.
The scalar field $z$ satisfies \eqref{eq:IBVP_uf} (since $\psi(\bx,0)=0$ implies $\der zt(\bx,0)=0$) with 
$f=\der{\psi}{t}-\nabla\cdot\bPhi$ 
so that 
\begin{align}
\nonumber
&z\in C^0\big([0,T];H^1\OO\big)\cap C^1\big([0,T];L^2\OO\big) \qquad\text{and}\\
&\N{z}_\LtQ+\N{\nabla z}_{\LtQ^3}+\N{\der zt}_\LtQ
\le C_S
\N{\der{\psi}{t}-\nabla\cdot\bPhi}_{L^2(0,T;L^2\OO)}.
%
\label{eq:zH1bound}
\end{align}

In the following will make make use of
\begin{equation}\label{eq:trickPoincare}
w(0)=0 \; \Rightarrow\; \N{w}^2_{L^2(0,T)}\le \frac{T^2}{2}\N{\der wt}_{L^2(0,T)}^2.
\end{equation}

From the PDEs in \eqref{eq:zzIBVP},
the regularity of $z$ \eqref{eq:zH1bound},
and the initial condition $\bzeta(\bx,0)=\bzero$, we have also 
\begin{align*}
\bzeta\in C^0\big([0,&T];\Hdiv\big)\cap C^1\big([0,T];L^2\OO^3\big)\qquad\text{and}\\
\N{\bzeta}_{\LtQ^3}&+\N{\nabla \cdot\bzeta}_\LtQ+\N{\der \bzeta t}_{\LtQ^3}
\\
&\overset{\eqref{eq:trickPoincare}}\le\N{\nabla \cdot\bzeta}_\LtQ+C_S\N{\der \bzeta t}_{\LtQ^3}\\
&=\N{\psi-c^{-2}\der zt}_\LtQ+C_S\N{\bPhi-\nabla z}_{\LtQ^3}\\
&\le \N{\psi}_\LtQ+C_S\N{\bPhi}_{\LtQ^3}+\N{c^{-2}\der zt}_\LtQ+C_S\N{\nabla z}_{\LtQ^3}\\
&\overset{\eqref{eq:zH1bound}}\le
C_S\Big(\N{\psi}_\LtQ+\N{\bPhi}_{\LtQ^3}
+\N{\der{\psi}{t}-\nabla\cdot\bPhi}_{L^2(0,T;L^2\OO)}\Big).
\end{align*}
%
We show that, for any $t\in(0,T)$, $\bzeta(\cdot,t)\in \bX_N(\Omega\times\{t\})$, where
\[
 \bX_N(\Omega\times\{t\}):= \big\{\bxi\in H(\div;\Omega\times\{t\})\cap H(\curl;\Omega\times\{t\}),\; \bxi\times\bn^x_\Omega=\bzero\ 
 \text{on\ } \partial\Omega\times\{t\}\big\}.                                      
\]
We already know that $\bzeta(\cdot,t)\in H(\div;\Omega\times\{t\})$ for all $t\in(0,T)$.
Since $\bPhi$ is supported inside $Q$, from the first PDE $\nabla z+\partial\bzeta/\partial t=\bPhi$ and the boundary condition $z=0$ on $\deO\times I$ 
in \eqref{eq:zzIBVP}, we have $\der{}t\bzeta\times\bn^x_\Omega=-\nabla z\times\bn^x_\Omega=-\nabla_T z\times\bn^x_\Omega=\bzero$ on $\deO\times I$, where $\nabla_T$ is the 
tangential gradient on $\deO\times I$, i.e.\ $\nabla_T z=\nabla z-\bn^x_\Omega(\bn^x_\Omega\cdot\nabla z)$.
Since $\bzeta(\bx,0)=\bzero$ for $\bx\in\deO$, we deduce $\bzeta\times\bn^x_\Omega=\bzero$ on $\partial\Omega\times I$.
In order to show that also $\curl\bzeta(\cdot,t)$ is bounded in $L^2(\Omega\times\{t\})$, we observe that,
%
%
from $\nabla\times\der{}t\bzeta=\nabla\times\bPhi$ and the initial condition for $\bzeta$, we have
$$
\N{\nabla\times\bzeta}_{L^2(\Omega\times\{t\})}=
\N{\int_0^t\nabla\times\bPhi(\cdot,s)\di s}_{L^2(\Omega)}.
%
$$
%
According to the assumptions stipulated in \S\ref{s:IBVP}, $\Omega$ is a Lipschitz polyhedron, thus by \cite[Proposition~3.7]{ABD98} there exists $\epsilon_\Omega>0$ such that
$$
\bX_N(\Omega\times\{t\})\subset H^{1/2+\epsilon_\Omega}(\Omega\times\{t\})^3 \qquad\forall t\in I
$$
with continuous inclusion.
Thus, for this $\epsilon_\Omega$ we have 
\begin{align*}
\N{\bzeta}_{H^{1/2+\epsilon_\Omega}(\Omega\times\{t\})^3}
&\le C_S 
\big(\N{\bzeta}_{H(\dive;\Omega\times\{t\})}+\N{\bzeta}_{H(\curl;(\Omega\times\{t\})}\big)
\qquad \forall t\in I,
\end{align*}
where $C_S$, here and in the following, depends on $\Omega$ also through $\epsilon_\Omega$.
%
Collecting the bounds on $\bzeta$ we have that
\begin{align}
\nonumber
\N{\bzeta}_{L^2(0,T;H^{1/2+\epsilon_\Omega}(\Omega)^3)}
&\le C_S
\big(
\N{\bzeta}_{\LtQ^3}
+\N{\nabla\cdot\bzeta}_\LtQ
+\N{\nabla\times\bzeta}_{\LtQ^3}
\big)
\\
&\le C_S 
\N{(\psi,\bPhi)}_\bX.
\label{eq:zetaHepsBound}
\end{align}

We use the previous bounds in order to control the traces of $z$ and $\bzeta\cdot\bn^x_F$ on the relevant time-like faces.
These traces are clearly bounded by the $H^1$ and $H^{1/2+\epsilon_\Omega}$ volume norm 
(which in turn are controlled by the $\bX$ norm of the data from \eqref{eq:zH1bound} and \eqref{eq:zetaHepsBound}). 
On the other hand, the trace inequality is precisely the point where the bounding constants depend on the mesh, so this is the point where we use the definition of $\alpha$ and $\beta$ 
in order to ensure that negative powers of the element sizes contain only the global meshwidth $h_\calT$.

Using the generalisation to $n=3$ of the trace inequalities \cite[eq.~(19)--(20)]{HMP13}
(see also Lem\-ma~\ref{lem:Trace} below)
we have, for some $C_\rho$ only depending on $\rho$ and $\epsilon_\Omega$,
\begin{align*}
&\Norm{\beta^{-1/2}z}^2_{L^2(\Ftime)}\\
&\le
\sum_{K\in\calT_h}\int_{t_K^-}^{t_K^+}\N{\beta^{-1/2}z}^2_{L^2(\deK_0\cap\Ftime)}\di t\\
&\le C_\rho
\sum_{K\in\calT_h}\int_{t_K^-}^{t_K^+}
\N{\beta^{-1}}_{L^\infty(\deK\cap\Ftime)}
\Big(\diam(K_0)^{-1}\N{z}^2_{L^2(K_0)}
+\diam(K_0)\N{\nabla z}^2_{L^2(K_0)^3}
\Big)\di t\\
&\le C_\rho \frac{\ttb_*^{-1}
}{h_\calT}
\sum_{K\in\calT_h}\frac1{c_{|_K}}\int_{t_K^-}^{t_K^+}
\Big(\N{z}^2_{L^2(K_0)}+\diam(K_0)^2\N{\nabla z}^2_{L^2(K_0)^3}\Big)\di t\\
&\le C_SC_\rho \frac{\ttb_*^{-1}}{h_\calT}
\left(\N{z}^2_{L^2(Q)}+\N{\nabla z}^2_{L^2(Q)^3}\right)
\\
&\overset{\eqref{eq:zH1bound}}\le C_S C_\rho 
\frac{\ttb_*^{-1}
}{h_\calT}
\N{\der{\psi}{t}-\nabla\cdot\bPhi}_{L^2(0,T;L^2\OO)}^2
\\&\le C_S C_\rho \frac{\ttb_*^{-1}}{h_\calT}\N{(\psi,\bPhi)}_\bX^2,
\end{align*}
and
\begin{align*}
&\Norm{\alpha^{-1/2}\bzeta\cdot \bn^x_F}^2_{L^2(\Ftime\cup\FD)}\\
&\le
\sum_{K\in\calT_h}
\int_{t_K^-}^{t_K^+}\N{\alpha^{-1/2}\bzeta\cdot \bn^x_K}^2_{L^2(\deK_0)}\di t\\
&\le C_\rho
\sum_{K\in\calT_h}\int_{t_K^-}^{t_K^+}
\N{\alpha^{-1}}_{L^\infty(\deK\cap(\Ftime\cup\FD))}\\
&\hspace{30mm}\cdot
\Big(\diam(K_0)^{-1}\N{\bzeta}^2_{L^2(K_0)^3}
+\diam(K_0)^{2\epsilon_\Omega}\abs{\bzeta}^2_{H^{1/2+\epsilon_\Omega}(K_0)^3}
\Big)\di t
\\
&\le C_\rho \frac{\tta_*^{-1}}{h_\calT}
\sum_{K\in\calT_h}c_{|_K}\int_{t_K^-}^{t_K^+}
\Big(\N{\bzeta}^2_{L^2(K_0)^3}
+\diam(K_0)^{1+2\epsilon_\Omega}\abs{\bzeta}^2_{H^{1/2+\epsilon_\Omega}(K_0)^3}
\Big)\di t\\
&
\le C_S C_\rho\frac{\tta_*^{-1}}{h_\calT}
\left(
\N{\bzeta}^2_{L^2(Q)^3}+\abs{\bzeta}^2_{L^2(0,T;H^{1/2+\epsilon_\Omega}(\Omega)^3)}
\right)
\\
&\overset{\eqref{eq:zetaHepsBound}}\le 
C_S C_\rho 
\frac{\tta_*^{-1}}{h_\calT}
\N{(\psi,\bPhi)}_\bX^2.
\end{align*}
The last two bounds give the desired result.
\end{proof}

Although in Proposition~\ref{prop:TimeFaceBounds} we did not track the dependence of the bounding constants on the wave speed $c$, we have defined the Trefftz-DG numerical flux parameters $\alpha$ and $\beta$ using the local wave speeds $\ttc_\pm$. 
This ensures that $\tta$ and $\ttb$ are dimensionless, while $\beta$ and $\alpha^{-1}$ maintain the dimensions of speeds.

Combining the bounds on the different terms of the $|\cdot|_{\Fh}$
seminorm obtained in Propositions~\ref{prop:TimeFaceBounds} and
\ref{prop:DualStability} with Proposition~\ref{prop:Duality},
immediately gives an error estimate in $\bX^*$ norm, dual to
\eqref{eq:XNorm3D}.

In order to control a more concrete norm of the error, we introduce a slightly weaker norm, defined using Bochner spaces and standard Sobolev spaces with negative exponents ($H^{-1}(I):=H^1_0(I)^*$ and $H^{-1}\OO^3:=(H^1_0\OO^3)^*$).


\begin{cor}\label{cor:Hminus1}
Under the assumptions of Proposition \ref{prop:TimeFaceBounds}, the following error bound holds:
\begin{align*}
&\N{v-v\hp}_{H^{-1}(0,T;L^2\OO)}+\N{\bsigma-\bsigma\hp}_{L^2(0,T;H^{-1}\OO^3)}\\
&\qquad\le  \Cstab (1+C_c) 
 \inf_{(w,\btau)\in\bVp\Th} \Tnorm{(v-w,\bsigma-\btau)}\DGp,
\end{align*}
where $C_c$ is as in \eqref{eq:Continuity} and 
$\Cstab=(\tCstab+C_{S,\rho}^{1/2}(\tta_*^{-1/2}+\ttb_*^{-1/2})h_\calT^{-1/2})$,
with
$\tCstab$ as in \eqref{eq:CstabNoTimeLike} and 
$C_{S,\rho}$ depending only on $\Omega$, $T$, $c$ and $\rho$.
%
\end{cor}
\begin{proof}
Propositions \ref{prop:DualStability} and \ref{prop:TimeFaceBounds} ensure that Assumption~\ref{ass:DualStability} holds for the space $\bX$ defined in Proposition~\ref{prop:TimeFaceBounds}.
Thus Proposition~\ref{prop:Duality} applied to the Galerkin error, together with the quasi-optimality \eqref{eq:QuasiOpt}, gives
\begin{align*}
&\N{(v-v_{hp},\bsigma-\bsigma_{hp})}_{\bX^*} \\
&\le  (1+C_c) \Big(\tCstab+C_{S,\rho}^{1/2} (\tta_*^{-1/2}+\ttb_*^{-1/2})h_\calT^{-1/2}\Big)
\inf_{(w,\btau)\in\bVp\Th} \Tnorm{(v-w,\bsigma-\btau)}\DGp.
\end{align*}
We define
\begin{align*}
\bY:=&H^1_0(0,T;L^2\OO)\times L^2(0,T;H^1_0\OO^3),\\
\N{(\psi,\bPhi)}^2_{\bY}:=&
\N{\psi}^2_{H^1(0,T;L^2\OO)}+\N{\bPhi}^2_{L^2(0,T;H^1_0\OO^3)}
\\
:=&\N{\psi}^2_{L^2(Q)}+\N{\der\psi t}^2_{L^2(Q)}
+\N{\bPhi}^2_{L^2(Q)^3}+\N{\nabla\cdot\bPhi}^2_{L^2(Q)}+\frac{T^2}2\N{\nabla\times\bPhi}^2_{L^2(Q)^3}
\end{align*}
(recall that $H^1_0\OO^3=\{\bPhi\in L^2\OO^3,\nabla\cdot\bPhi\in
L^2\OO,\nabla\times\bPhi\in
L^2\OO^3,\bPhi\cdot\bn_\Omega^x=0,\bPhi\times\bn_\Omega^x=\bzero\oon\deO\}$;
see e.g.\ \cite[Theorem~2.5]{ABD98}).
By \eqref{eq:XNorm3D} and \eqref{eq:trickPoincare}, $\bY\subset\bX$ and $\N{(\psi,\bPhi)}_{\bX}\le\N{(\psi,\bPhi)}_{\bY}$.
This allows to control the Trefftz-DG error in the desired norm: for all $(w,\btau)\in \bX^*$
\begin{align*}
\N{w}&_{H^{-1}(0,T;L^2\OO)}+\N{\btau}_{L^2(0,T;H^{-1}\OO^3)}\\
&=
\sup_{0\ne\psi\in H^1_0(0,T;L^2\OO)}\frac{\int_Q w\psi\di x\di t}{\N{\psi}_{H^1(0,T;L^2\OO)}}
+\sup_{\bzero\ne\bPhi\in L^2(0,T;H^1_0\OO^3)}
\frac{\int_Q\btau\cdot\bPhi\di x\di t}{\N{\bPhi}_{L^2(0,T;H^1_0\OO^3)}}\\
&=
\sup_{\substack{\psi\in  H^1_0(0,T;L^2\OO) \\ \N{\psi}_{H^1(0,T;L^2\OO)}=\frac1{\sqrt2}}}
\int_Q w\psi\di x\di t
+\sup_{\substack{\bPhi\in L^2(0,T;H^1_0\OO^3) \\ \N{\bPhi}_{L^2(0,T;H^1_0\OO^3)}=\frac1{\sqrt2}}}
\int_Q\btau\cdot\bPhi\di x\di t\\
&=
\sup_{\substack{(\psi,\bPhi)\in  \bY \\ 
\N{\psi}_{H^1(0,T;L^2\OO)}=\N{\bPhi}_{L^2(0,T;H^1_0\OO^3)}=\frac1{\sqrt2}}}
\frac{\int_Q(w\psi+\btau\cdot\bPhi)\di x\di t}{\N{(\psi,\bPhi)}_\bY}\\
&\le
\sup_{(0,\bzero)\ne(\psi,\bPhi)\in  \bY}
\frac{\int_Q(w\psi+\btau\cdot\bPhi)\di x\di t}{\N{(\psi,\bPhi)}_\bY}\\
&\le
\sup_{(0,\bzero)\ne(\psi,\bPhi)\in  \bX}
\frac{\int_Q(w\psi+\btau\cdot\bPhi)\di x\di t}{\N{(\psi,\bPhi)}_\bX}
=\N{(w,\btau)}_{\bX^*}.
\end{align*}
We conclude
by choosing $w=v-v\hp$, $\btau=\bsigma-\bsigma\hp$ and combining with the previous bound.
%
\end{proof}
\am{Under the assumptions of Corollary \ref{cor:Hminus1}, if all space--time mesh elements have ``length in time'' (i.e.\ $t_K^+-t_K^-$) proportional to $h_t$, then the value of $\Cstab(1+C_c)\approx (h_t^{-1/2}+h_\calT^{-1/2})$, which means that the convergence of the Galerkin error in the $H^{-1}(0,T;L^2\OO)\times L^2(0,T;H^{-1}\OO^3)$ norm is half order slower than the best-approximation error in $\Tnorm{\cdot}\DG$ norm.
}

\am{We stress once again that in the cases
{\em i)} $n=1$ and rectangular meshes with interfaces parallel to the space--time axes, 
and {\em ii)} $n\ge 1$, Robin boundary conditions only, and meshes with no time-like mesh interfaces,
adjoint stability holds with $\bX=L^2(Q)^{1+n}$,
and thus error estimates in the $L^2(Q)^{1+n}$ norm follow. 
The two cases are described in \cite{SpaceTimeTDG} and Corollary~\ref{cor:NoTimeLike}, respectively.
The identification of other situations where this holds true is an open problem.}

\section{Polynomial Trefftz spaces}\label{s:Spaces}

So far we have not specified any discrete (test and trial) space $\bVp \Th$: the only condition we imposed is the Trefftz property $\bVp \Th\subset\bT\Th$.
For time-harmonic problems, non-trivial polynomial Trefftz spaces do not exist and typical bases are constituted by plane waves or Fourier--Bessel functions (several other basis have been developed, see \cite[\S3]{TrefftzSurvey} for a detailed overview).
In the current time-domain setting, however, one has more freedom and can choose piecewise-polynomial Trefftz spaces.
This is due to the fact that the PDE we discretise is {\em homogeneous}, in the sense that all terms appearing in it are derivatives of the same order, thus polynomial solutions are admitted and give high-order approximation properties (see Lemma~\ref{lem:Taylor} below).

We fix some notation.
Given $p\in\IN_0:=\IN\cup\{0\}$, $k\in\IN$, $D\subset\IR^k$, we denote by $\IP^p(D)$ the space of polynomial of degree at most $p$ in $k$ variables (when we write $k=n+1$ the last one is understood as the time variable) restricted to $D$.
We use standard multi-index notation: for $\malpha\in\IN_0^n$, we denote
$|\malpha|=\alpha_1+\cdots+\alpha_n$,
$D^\malpha \varphi=\frac{\partial^{|\malpha|}\varphi}{\partial x_1^{\alpha_1}\cdots\partial x_n^{\alpha_n}}$,
$\bx^\malpha=x_1^{\alpha_1}\cdots x_n^{\alpha_n}$,
$\be_m:=(0,\ldots, 0,1,0,\dots,0)\in\IN_0^n$ with 1 in the $m$th entry, and
$\binom kj=\frac{k!}{j!(k-j)!}$ for $j\le k\in\IN_0$. 
For a space--time field $\varphi$, $D^{\malpha,\alpha_t} \varphi=\frac{\partial^{|\malpha|+\alpha_t}\varphi}{\partial x_1^{\alpha_1}\cdots\partial x_n^{\alpha_n}\partial t^{\alpha_t}}$.
For $s>0$, the broken Sobolev spaces on the mesh $\calT_h$ are denoted by
$H^s\Th:=\{w\in L^2(Q),$ s.t.\ $w_{|_K}\in H^s(K)\;\forall K\in\calT_h\}$.

The simplest discrete Trefftz space is
\begin{align*}
\bVp \Th=\ITp\Th:=
\prod_{K\in\calT_h}\IT^{p_K}(K),\quad\text{where}\;
\IT^{p_K}(K):=\bT(K)\cap\IP^{p_K}(\IR^{n+1})^{1+n},\quad p_K\in\IN_0,
\end{align*}
i.e.\ the space containing the fields $(w,\btau)$ that in each mesh
element $K$ are solution of the wave equations and are polynomials of
degree at most $p_K$.

If the specific first-order IBVP at hand comes from a second-order one, i.e.\ $v=\der Ut$ and $\bsigma=-\nabla U$ for some scalar $U$ solution of $-\Delta U+c^{-2}\der{{^2}}{t^2}U=0$, then one can use a slightly smaller discrete space
\begin{align}\label{eq:ST}
\bVp \Th=\IWp\Th:&=
\prod_{K\in\calT_h}\IW^{p_K}(K),\quad\text{where}\quad \\
\IW^{p_K}(K):&=\Big\{(v,\bsigma)\in\IT^{p_K}(K): v=\der Ut, \bsigma=-\nabla U, \; \text{for some}\;U\in\IP^{p_K+1}(\IR^{n+1})\Big\}.
\nonumber
\end{align}
(Note that if $(\der Ut,-\nabla U)\in \bT(K)$, then $-\Delta U+c^{-2}\der{{^2}}{t^2}U=0$ follows by the Trefftz property.)
For $n\ge 2$, not all elements of $\IT^{p_K}(K)$ belong to $\IW^{p_K}(K)$: 
e.g., 
$(0, (x_2,0,\ldots,0))\in \IT^1(K)\setminus\IW^1(K)$;
see also Remark~\ref{rem:TvsS} below.

In the next two subsections, we consider the two local polynomial Trefftz spaces $\IT^p(K)$ and $\IW^p(K)$, respectively. 
For each of them we describe a simple basis and derive high-order approximation properties in the meshwidth $h$.

\begin{remark}
If $n=1$, then the two spaces coincide: $\IW^{p_K}(K)=\IT^{p_K}(K)$.
The approximation bounds proved in the next section guarantee
$h$-convergence only; we proved sharper $p$-convergence bounds in
\cite[\S5.3--6]{SpaceTimeTDG} with different techniques that do not
easily extend to higher space dimensions.
\end{remark}


\subsection{The full polynomial Trefftz space \texorpdfstring{$\IT^p(K)$}{Tp(K)}}
\label{s:Tp}

\subsubsection{A basis of \texorpdfstring{$\IT^p(K)$}{Tp(K)}}
\label{s:TpBasis}

A basis for $\IT^p(K)$ can be constructed by ``evolving'' in time polynomial initial conditions.
Given any basis $\{\widetilde b_\ell(\bx)\}_{\ell=1,\ldots,\binom{p+n}n}$ of $\IP^p(\IR^n)$, a basis for $\IT^p(K)$ is given by
\begin{align*}
\Bigg\{\bb_{\ell,j}(\bx,t)\in\IT^p(K) \text{ such that }
\begin{aligned}
&\bb_{\ell,0}(\bx,0)=\big(\widetilde b_\ell(\bx), \bzero\big),\\
&\bb_{\ell,j}(\bx,0)=\big(0, \widetilde b_\ell(\bx)\be_j\big), \; j=1,\ldots, n
\end{aligned}
\Bigg\}_{\ell=1,\ldots,\binom{p+n}n;\; j=0,\ldots,n}.
\end{align*}
As a consequence
\begin{align}\label{eq:dimTp}
\dim\big(\IT^p(K)\big)
=(n+1)\binom{p+n}n.
\end{align}
We note that $\dim(\IT^p(K))=\calO_{p\to\infty}(p^n)$, while the full (vector-valued) polynomial space has much larger dimension
$\dim(\IP^p(\IR^{n+1})^{1+n})=(p+n+1)\binom{p+n}n=\calO_{p\to\infty}(p^{n+1})$.

To compute explicitly the basis elements $\bb_{\ell,j}$ from $\widetilde b_\ell$,
we expand in monomials the general polynomial $(v,\bsigma)\in \IP^p(\IR^{n+1})^{1+n}$:
\begin{align*}
\begin{aligned}
v(\bx,t)&=\sum_{\substack{k\in\IN_0,\malpha\in\IN_0^n\\ k+|\malpha|\le p}}
a_{v,k,\malpha}\bx^\malpha t^k,
\\
\bsigma(\bx,t)&=\Bigg(
\sum_{\substack{k\in\IN_0,\malpha\in\IN_0^n\\ k+|\malpha|\le p}}
a_{\sigma_1,k,\malpha}\bx^\malpha t^k,
\ldots,
\sum_{\substack{k\in\IN_0,\malpha\in\IN_0^n\\ k+|\malpha|\le p}}
a_{\sigma_n,k,\malpha}\bx^\malpha t^k\Bigg),
\end{aligned}
\quad
\begin{aligned}
&\text{for } a_{v,k,\malpha},\\
&a_{\sigma_1,k,\malpha},\ldots,a_{\sigma_n,k,\malpha}\in\IR.
\end{aligned}
\end{align*}
Then  $(v,\bsigma)\in\IT^p(K)$ if and only if the coefficients satisfy the recurrence relations
\begin{align}\label{eq:RecurrenceTp}
\begin{aligned}
a_{v,k,\malpha}&=-\frac{c^2}k \sum_{m=1}^n(\alpha_m+1)a_{\sigma_m,k-1,\malpha+\be_m},\\
a_{\sigma_m,k,\malpha}&=-\frac1k (\alpha_m+1)a_{v,k-1,\malpha+\be_m},
\end{aligned}
\qquad k=1,\ldots,p,\quad |\malpha|\le p-k, 
  \quad m=1,\ldots,n.
\end{align}
These formulas allow to compute all coefficients $a_{v,k,\malpha},a_{\sigma_1,k,\malpha},\ldots,a_{\sigma_n,k,\malpha}$ starting from those with index $k=0$, which correspond to the values at $t=0$:
\begin{align*}
v(\bx,0)&=\sum_{\substack{\malpha\in\IN_0^n, |\malpha|\le p}}a_{v,0,\malpha}\bx^\malpha,\;\\
\bsigma(\bx,0)&=
\Bigg(\sum_{\substack{\malpha\in\IN_0^n, |\malpha|\le p}}a_{\sigma_1,0,\malpha}\bx^\malpha,
\ldots,
\sum_{\substack{\malpha\in\IN_0^n, |\malpha|\le p}}a_{\sigma_n,0,\malpha}\bx^\malpha\Bigg).
\end{align*}
For all $\widetilde b_\ell(\bx)=\sum_{|\malpha|\le p} a_\malpha^{(\ell)}\bx^\malpha$, 
the corresponding $n+1$ space--time basis elements $\bb_{\ell,0},\ldots,\bb_{\ell,n}$ 
are computed using the recurrence \eqref{eq:RecurrenceTp} starting from the coefficients 
$a_\malpha^{(\ell)}$ 
(i.e.\ 
$a_{v,0,\malpha}=a_\malpha^{(\ell)}$, $a_{\sigma_j,0,\malpha}=0$, $j=1,\ldots,n$ for $\bb_{\ell,0}$ 
and 
$a_{v,0,\malpha}=0$, $a_{\sigma_j,0,\malpha}=\delta_{j,j'}a_\malpha^{(\ell)}$ for $\bb_{\ell,j'}$, $j,j'=1,\ldots,n$).

If the space basis functions $\widetilde b_\ell$ are homogeneous polynomials, their corresponding space--time basis elements $\bb_{\ell,j}$ are homogeneous as well.


Assume that a Trefftz basis is constructed in reference coordinates $(\widehat{\bx},\widehat{t})$ for a reference velocity, say, equal to $1$.
A scaled and centred Trefftz basis in an element $K$ in the space of the physical coordinates
$(\bx,t)$, where the material velocity is $c$, can be easily derived.
Let $h_K$ be a characteristic dimension of $K$ (e.g.\ the ``anisotropic'' diameter introduced in
Assumption~\ref{as:StarShaped} below), and $(\bx_K,t_K)$ be a point in $K$. 
Given the change of variables $(\bx,t)=(h_K\widehat{\bx},h_Kc^{-1}\widehat{t})+(\bx_K,t_K)$,
if $(\widehat{v},\widehat{\bsigma})$ is Trefftz with velocity $1$, then $(v,\bsigma)$ defined as
\[
v(\bx,t)=c\,\widehat{v}(\widehat{\bx},\widehat{t}),\qquad
\bsigma(\bx,t)=\widehat{\bsigma}(\widehat{\bx},\widehat{t})
\]
is Trefftz with velocity $c$.
In this way one can use a single reference basis to define scaled bases on each $\IT^p(K)$.

\subsubsection{Approximation theory for \texorpdfstring{$\IT^p(K)$}{Tp(K)}}
\label{s:TpApprox}


To study the approximation properties of the spaces $\ITp\Th$, we begin with a simple general lemma stating that, for linear {\em homogeneous} PDEs (i.e.\ such that all their terms are derivatives of the same  order) with constant coefficients, Taylor and averaged Taylor polynomials of solutions are Trefftz.
This implies that the Bramble-Hilbert lemma can be proved as in~\cite{BRS94,Dur83} and the orders of $h$-convergence for Trefftz spaces are the same as those for full polynomial spaces.


\begin{lemma}\label{lem:Taylor}
Let $\Upsilon\subset\IR^N$, $2\le N\in\IN$, be an open bounded set, with diameter $h$, star-shaped with respect to the ball $B:=B_{\rho h}(\by)$ centred at $\by\in\Upsilon$ and with radius $\rho h$, $0<\rho\le 1/2$.

Let $\calL$ be a linear differential operator with constant coefficients defined on $\IR^{N'}$-valued  fields in $\Upsilon$ and with values in $\IR^{N''}$, with $N',N''\in\IN$.
Assume that $\calL$ is homogeneous, in the sense that all its terms have degree $d\in\IN$:
$$
(\calL\bu)_{i''}=
\sum_{i'=1,\ldots,N'}\sum_{\malpha\in\IN_0^N,\; |\malpha|=d }
a_{\malpha,i',i''}D^\malpha u_{i'}
\qquad i''=1,\ldots,N'',\quad \bu\in C^d(\Upsilon)^{N'}.
$$
Let $\bu\in L^2(\Upsilon)^{N'}$ be a (distributional) solution of $\calL\bu=\bzero$.
Fix $m\in \IN$. 
\begin{enumerate}[(i)]
\item \label{a1}
If $\bu\in C^{m-1}(\Upsilon)^{N'}$, then the (vector-valued) Taylor polynomial $T^m_\by[\bu]$ 
of order $m$ (and polynomial degree at most $m-1$) centred at $\by$ 
\[
T^m_\by[\bu](\bx):=\sum_{|\malpha|< m}\frac1{\malpha!}D^\malpha \bu(\by)(\bx-\by)^\malpha
\]
satisfies $\calL (T^m_\by[\bu])=\bzero$.
\item \label{a2}
If $\bu\in H^{m-1}(\Upsilon)^{N'}$, then the averaged Taylor polynomial 
\begin{align*}
Q^m[\bu](\bx)
:=\frac1{|B|}\int_B T^m_\by[\bu](\bx)\di \by
\end{align*}
is a polynomial of degree at most $(m-1)$ and satisfies $\calL (Q^m[\bu])=\bzero$.
\item \label{a3}
If $\bu\in H^m(\Upsilon)^{N'}$, then the averaged Taylor polynomial $Q^m[\bu]$ satisfies the approximation estimate
\begin{align*}
\abs{\bu- Q^{m}[\bu]}_{H^j(\Upsilon)^{N'}}\leq
2\binom{N+j-1}{N-1}
\frac{N^{m-j}}{(m-j-1)!} 
\frac{h^{m-j}}{\rho^{N/2}}\abs{\bu}_{H^m(\Upsilon)^{N'}}
\qquad 0\le j\le m-1.
\end{align*}
\end{enumerate}
\end{lemma}

\begin{proof}
\eqref{a1} and \eqref{a2} follow from the identities $D^\mbeta T^m_\by[\bu]=T^{m-|\mbeta|}_\by[D^\mbeta\bu]$ and 
$D^\mbeta Q^m[\bu]=Q^{m-|\mbeta|}[D^\mbeta\bu]$, for $\mbeta\in\IN_0^N$, 
$|\mbeta|<m$, see e.g.\ \cite[eq.~(3.5)]{AndreaPhD}.
For \eqref{a3}, applying the main assertion of \cite{Dur83}
  componentwise gives, for all $0\le j\le m-1$,
\begin{align*}
\abs{\bu- Q^{m}[\bu]}_{H^j(\Upsilon)^{N'}}\leq
2\binom{N+j-1}{N-1}(m-j)\Big(\sum_{|\malpha|=m-j}(\malpha!)^{-2}\Big)^{1/2}
\frac{h^{m-j}}{\rho^{N/2}}\abs{\bu}_{H^m(\Upsilon)^{N'}}.
\end{align*}
Bounding the sum of the factorials as in \cite[(B.9)]{AndreaPhD} gives 
the bound in \eqref{a3}.
\end{proof}

We introduce an anisotropic Sobolev seminorm in an open $K\subset Q$: for 
$(v,\bsigma)\in H^j(K)^{1+n}$, $j\in \IN_0$,
\begin{align}
\abs{(v,\bsigma)}_{H^j_c(K)}^2:
=&\sum_{\substack{\malpha\in\IN_0^n, \alpha_t\in\IN_0\\|\malpha|+\alpha_t=j}}
\bigg(\N{c^{-1/2-\alpha_t}D^{\malpha,\alpha_t} v}_{L^2(K)}^2
+\N{c^{1/2-\alpha_t}D^{\malpha,\alpha_t} \bsigma}_{L^2(K)^n}^2\bigg).
\label{eq:HjcNorm}
\end{align}
In particular,
$\N{(v,\bsigma)}_{L^2_c(K)}^2:=\abs{(v,\bsigma)}_{H^0_c(K)}^2
=\N{c^{-1/2}v}_{L^2(K)}^2+\N{c^{1/2}\bsigma}_{L^2(K)^n}^2$.

We will make the following assumption on the shape of the mesh elements.
\begin{assum}\label{as:StarShaped}
Let $K\subset \IR^{n+1}$ be an open, bounded, Lipschitz set.
We denote its  ``an\-i\-so\-trop\-ic diameter'' by
$$h_K:=\sup_{(\bx,t),(\by,s)\in K}\Big\{\big(
\abs{\bx-\by}^2
+c^2(t-s)^2\big)^{1/2}\Big\},$$
and we assume that $K$ is star-shaped with respect to the ellipsoid 
$$\big\{(\bx,t)\text{ such that\ }|\bx-\bx_K|^2+c^2(t-t_K)^2<\rho^2h_K^2\big\},$$ 
for some $(\bx_K,t_K)\in K$ and $0<\rho\le 1/2$.
\end{assum}
As a simple example, if an element $K\in\calT_h$ is Cartesian product of $n$ (space) interval of length $h_K^x$ and one (time) interval of length $h_K^x/c$, 
then it satisfies Assumption~\ref{as:StarShaped} with $h_K=\sqrt{n+1}\,h_K^x$ and $\rho=1/(2\sqrt{n+1})$.

\begin{cor}\label{cor:ApproxTp}
For $p,s\in\IN_0$ and for constant $c>0$, let $(v,\bsigma)\in \bT(K)\cap H^{s+1}(K)^{1+n}$ for a space--time domain $K$ as in Assumption~\ref{as:StarShaped}.
Then, there exists $(w\hp,\btau\hp)\in \IT^p(K)$ such that, for all $0\le j\le m:=\min\{p,s\}$,
\begin{align*}
\abs{(v-w\hp,\bsigma-\btau\hp)}_{H^j_c(K)}\leq
2 \binom{n+j}n
\frac{(n+1)^{m+1-j}}{(m-j)!} 
\frac{h_K^{m+1-j}}{\rho^{(n+1)/2}}\abs{(v,\bsigma)}_{H^{m+1}_c(K)}
.
\end{align*}
\end{cor}
\begin{proof}
The scaled fields $\widetilde v(\bx,s)=c^{-1}v(\bx,s/c)$, 
$\widetilde \bsigma(\bx,s)=\bsigma(\bx,s/c)$ satisfy the wave equations with unit speed
$\nabla \widetilde v + \der{\widetilde \bsigma}s=0$, 
$\nabla \cdot\widetilde \bsigma + \der{\widetilde v}s=0$ in $\widetilde K=\{(\bx,s)$ s.t. $(\bx, s/c)\in K\}$, which is star-shaped with respect to a ball with radius $\rho h_K$.
By Lemma~\ref{lem:Taylor}, with $N=N'=N''=n+1$,
\begin{align*}
\abs{(\widetilde v,\widetilde \bsigma)-Q^{p+1}[(\widetilde v,\widetilde \bsigma)]}_{H^j(\widetilde K)^{1+n}}
\leq 2 \binom{n+j}n
\frac{(n+1)^{m+1-j}}{(m-j)!} 
\frac{h_K^{m+1-j}}{\rho^{(n+1)/2}}\abs{(\widetilde v,\widetilde\bsigma)}_{H^{m+1}(\widetilde K)^{1+n}}
.
\end{align*}
From \eqref{eq:HjcNorm} and a simple scaling, we have
$\abs{(v,\bsigma)}_{H^{m+1}_c(K)}
=\abs{(\widetilde v,\widetilde\bsigma)}_{H^{m+1}(\widetilde K)^{1+n}}$, 
and similarly for the norm at the left-hand side, from which the assertion follows.
\end{proof}


\begin{lemma}\label{lem:Trace}
Let $\Upsilon\subset\IR^N$ be as in Lemma~\ref{lem:Taylor} and $u\in H^1(\Upsilon)$.
Then, for all $a>0$, we have the following trace estimate:
\begin{align}\label{eq:lem:Trace}
\N{u}_{L^2(\partial\Upsilon)}^2\le 
\frac{N+a}{\rho h}\N{u}^2_{L^2(\Upsilon)}
+\frac{h}{\rho a}\N{\nabla u}_{L^2(\Upsilon)^N}^2.
\end{align}
\end{lemma}
\begin{proof}
We slightly extend the proof of \cite[Lemma~4.4]{HMP13}.
We assume for simplicity that the centre of the ball $B$ is
$\by=\bzero$. Recalling that $\bx\cdot\bn\ge\rho h$ on
$\partial\Upsilon$, where $\bn$ is the outward unit normal vector
  to $\partial\Upsilon$, we have
\begin{align*}
\N{u}_{L^2(\partial\Upsilon)}^2
&\le\frac1{\rho h}\int_{\partial\Upsilon}(\bx\cdot\bn) u^2\di S
\\&=\frac1{\rho h}\int_\Upsilon \dive(\bx u^2)\di \bx
\\&=\frac1{\rho h}\int_\Upsilon \big(N u^2+2(\bx\cdot\nabla u)u\big)\di \bx
\le\frac N{\rho h}\N{u}^2_{L^2(\Upsilon)}
+\frac2\rho\N{u}_{L^2(\Upsilon)}\N{\nabla u}_{L^2(\Upsilon)},
\end{align*}
and concluding by the arithmetic--geometric mean inequality.
\end{proof}
\begin{rem}
If $K$ is not star-shaped but can be decomposed in parts that are star-shaped with respect to some balls, a bound similar to \eqref{eq:lem:Trace} holds with at the denominator the radius of the smallest of those balls.
\end{rem}

For each $K\in\calT_h$ we denote the space-like and the time-like part of its boundary by 
$$
\deKspa:=\deK\cap(\Fspa\cup\FO\cup\FT),\qquad 
\deKtime:=\deK\cap(\Ftime\cup\FD\cup\FN\cup\FR)
.$$
We introduce some coefficients: for each $K\in\calT_h$
\begin{align}
\xi_K^{\mathrm{time}}:=\max\Big\{&
\N{2c\alpha}_{L^\infty(\deK\cap(\Ftime\cup\FD))}
+\N{c/\beta}_{L^\infty(\deK\cap(\Ftime\cup\FN))},\nonumber\\
&\N{2\beta/c}_{L^\infty(\deK\cap(\Ftime\cup\FN))}
+\N{1/(c\alpha)}_{L^\infty(\deK\cap(\Ftime\cup\FD))},\nonumber\\
&\N{(1-\delta)\vartheta}_{L^\infty(\deK\cap(\FR))},\quad
\N{\delta/\vartheta}_{L^\infty(\deK\cap(\FR))}
\Big\}
,\nonumber\\
\xi_K:=\max\Big\{&\xi_K^{\mathrm{time}},\quad
\N{n^t_K\big(2(1-\gamma)^{-1}+1\big)}_{L^\infty(\deKspa)}
\Big\}\ge 2\sqrt2.
\label{eq:xiK}
\end{align}

\begin{rem}
If $\alpha=\beta^{-1}=c^{-1}$, $\delta=\vartheta^2/(1+\vartheta^2)$, then
$\xi_K^{\mathrm{time}}=3$ 
and $\xi_K$ only depends on the maximal slope of the space-like faces of $K$ and on $c$. 
If, moreover, all faces of $K$ are aligned to the space--time axes,
then $\xi_K=3$ as well.
\end{rem}

\begin{theorem}\label{thm:ConvergenceTp}
Assume there exists $0<\rho<1/2$ such that all mesh elements $K\in\calT_h$ satisfy Assumption~\ref{as:StarShaped}.
Let $(v,\bsigma)
$ and $(\Vhp,\Shp)
$ be the solutions of the IBVP~\eqref{eq:IBVP} and of the Trefftz-DG formulation \eqref{eq:TDG} with $\bT\Th=\ITp\Th$, respectively.
For each element $K\in\calT_h$, assume local regularity $(v_{|_K},\bsigma_{|_K})\in H^{s_K+1}(K)^{1+n}$ for some $s_K\in\IN_0$ and define $m_K:=\min\{p_K,s_K\}$.
Then:
\begin{align}\nonumber
\Tnorm{(v-\Vhp,\bsigma-\Shp)}\DG
&\le (1+C_c)\rho^{-n/2-1}
\sum_{K\in\calT_h}\xi_K^{1/2} C_{(m_K,n)} h_K^{m_K+1/2}
\abs{(v,\bsigma)}_{H^{m_K+1}_c(K)},\\
\text{where }\quad&
C_{(m_K,n)}=
\begin{cases}
\frac{2\sqrt{n+3}(n+1)^{m_K+1}}{(m_K-1)!} 
& \text{if } m_K\ge1,\\
2\sqrt{n+2}(n+1)+1& \text{if } m_K=0,\\
\end{cases}
\label{eq:thm:ConvergenceTp}
\end{align}
$\xi_K$ as in~\eqref{eq:xiK}, and $C_c$ is the constant defined in~\eqref{eq:Continuity}.
\end{theorem}
\begin{proof}
From the definition \eqref{eq:DGnorm} of the $\Tnorm{\cdot}\DGp$ norm, for all $(w,\btau)\in H^1\Th^{1+n}$,
\begin{align*}
&\N{(w,\btau)}\DGp^2\\
&\le \sum_{K\in\calT_h}
\N{2\frac{|n^t_K|}{1-\gamma}+(1-\gamma)|n^t_K|
}_{L^\infty(\deKspa)}
\Big(\N{c^{-1}w}^2_{L^2(\deKspa)}+\N{\btau}^2_{L^2(\deKspa)^{n}}\Big)\\
&\hspace{50mm}
+\xi_K^{\mathrm{time}} c 
\Big(\N{c^{-1}w}^2_{L^2(\deKtime)}+\N{\btau}^2_{L^2(\deKtime)^{n}}\Big),
\end{align*}
where we used e.g.\ that $\jmp{w}_t^2\le2(n^t_{K_1})^2(w_{|_{K_1}}^2+w_{|_{K_2}}^2)$ on $\deK_1\cap\deK_2$.
We define the scaled fields $\widetilde w(\bx,s)=c^{-1}w(\bx,s/c)$, 
$\widetilde \btau(\bx,s)=\btau(\bx,s/c)$ in $\widetilde K=\{(\bx,s)$ s.t. $(\bx, s/c)\in K\}$.
Then the $L^2$ norms on a $n$-dimensional space--time face $F$ with unit normal $(\bn^\bx_F,n^t_F)$ scale as
$$
\N{w}^2_{L^2(F)} = c^2 \big(|n_F^t|^2+c^2|\bn_F^x|^2\big)^{-1/2} \N{\widetilde w}^2_{L^2(\widetilde F)},\;
\N{\btau}^2_{L^2(F)^n} 
= \big(|n_F^t|^2+c^2|\bn_F^x|^2\big)^{-1/2} \N{\widetilde \btau}^2_{L^2(\widetilde F)^n},
$$
where $\widetilde F=\{(\bx,s)$ s.t. $(\bx, s/c)\in F\}$.
Using this in the previous bound, recalling that $n^t_K=0$ on $\deKtime$ while  $c|\bn^x_F|=\gamma|n^t_F|$ and $\gamma\in[0,1)$ on $\deKspa$ by \eqref{eq:gamma}, we have
\begin{align}\nonumber
&\N{(w,\btau)}\DGp^2\\
&\le \sum_{K\in\calT_h}\xi_K
\Big(\N{\widetilde w}^2_{L^2(\partial\widetilde K)}
+\N{\widetilde \btau}^2_{L^2(\partial\widetilde K)^{n}}\Big)
\nonumber\\
&\overset{\eqref{eq:lem:Trace},a=1}\le 
\frac1\rho\sum_{K\in\calT_h}\xi_K
\bigg(\frac{n+2}{h_K}
\Big(\N{\widetilde w}^2_{L^2(\widetilde K)}+\N{\widetilde \btau}^2_{L^2(\widetilde K)^{n}}\Big)
+h_K
\Big(\abs{\widetilde w}^2_{H^1(\widetilde K)}+\abs{\widetilde \btau}^2_{H^1(\widetilde K)^{n}}\Big)
\bigg)
\nonumber\\
&=\frac1\rho\sum_{K\in\calT_h}\xi_K
\bigg(\frac{n+2}{h_K}\N{(w,\btau)}^2_{L^2_c(K)}+h_K \abs{(w,\btau)}^2_{H^1_c(K)}
\bigg).
\label{eq:DG+<H1}
\end{align}
If $m_K\ge1$ for all elements, we conclude by using the
quasi-optimality \eqref{eq:QuasiOpt} in Theorem~\ref{th:QO},
the bound~\eqref{eq:DG+<H1} with
$(w,\btau)=(v-w\hp,\bsigma-\btau\hp)$,
and the approximation bounds of Corollary~\ref{cor:ApproxTp} with $j=0,1$:
\begin{align*}
&\Tnorm{(v,\bsigma)-(\Vhp,\Shp)}\DG^2\\
&\le (1+C_c)^2 \frac{4(n+3)}{\rho^{n+2}} \sum_{K\in\calT_h}\xi_K  
\frac{(n+1)^{2m_K+2}}{((m_K-1)!)^2} h_K^{2m_K+1}\abs{(v,\bsigma)}_{H^{m_K+1}_c(K)}^2.
\end{align*}
The proof for $m_K=0$ follows in a similar way, using that 
$|(v-w\hp,\bsigma-\btau\hp)|_{H^1_c(K)}=|(v,\bsigma)|_{H^1_c(K)}$ 
for all $(w\hp,\btau\hp)\in\IT^0(K)$.
\end{proof}
The orders of convergence in the local meshwidth $h_K$ are optimal, $\Tnorm{\cdot}\DG$ \am{being} a trace norm.
Combining \eqref{eq:thm:ConvergenceTp} with the results of \S\ref{s:MeshIndependent}, one immediately obtains orders of convergence in mesh-independent norms.

\begin{rem}\label{rem:SimpleBoundTp}
If $\alpha=\beta^{-1}=c^{-1}$, $\delta=1/2$, $\vartheta=1$, 
the space-like faces are perpendicular to the time axis (i.e.\ $\gamma=0$),
$m_K\ge1$, and either $n=2$ or $n=3$, then one easily obtain $\xi_K\le3$, $C_c=2$ and $C_{(m_k,n)}\le 4 \sqrt6(n+1)^{m_K+1/2}/(m_K-1)!$, thus bound \eqref{eq:thm:ConvergenceTp} allows the simpler expression
\begin{align*}
\Tnorm{(v-\Vhp,\bsigma-\Shp)}\DG
&\le 36\sqrt2\rho^{-n/2-1}\sum_{K\in\calT_h} \frac{\big((n+1)h_K\big)^{m_K+1/2}}{(m_K-1)!}
\abs{(v,\bsigma)}_{H^{m_K+1}_c(K)}.
\end{align*}
\end{rem}

\begin{rem}
\am{Error bounds and orders of convergence in mesh-independent norms can be obtained by combining Theorem~\ref{thm:ConvergenceTp} with the results of \S\ref{s:MeshIndependent}, under suitable assumptions on the mesh and the coefficients.
For example, under the assumptions of both Proposition~\ref{prop:TimeFaceBounds} and Theorem~\ref{thm:ConvergenceTp}, if the mesh is partitioned in uniform time slabs (i.e.\ $\forall K\in\calT_h$, $K=K_0\times((j-1)T/N,jT/N)$ for $j=1,\ldots,N$) and $(v,\bsigma)\in H^{m+1}(Q)^{1+n}$, $p_K\ge m$, then 
\begin{align*}
&\N{v-v\hp}_{H^{-1}(0,T;L^2\OO)}+\N{\bsigma-\bsigma\hp}_{L^2(0,T;H^{-1}\OO^3)}
\le  C \,h_\calT^{m}\, |(v,\bsigma)|_{H^{m+1}_c(Q)},
\end{align*}
where $C$ only depends on $n,\Omega,T,c,\rho,\tta_*,\ttb_*,m$.
}
\end{rem}

The error bounds obtained in Theorem~\ref{thm:ConvergenceTp} coincide with those one might expect for a non-Trefftz space--time DG method using $\IP^{p_K}(K)^{1+n}$ as local discrete space.
However, the dimension of $\IT^{p_K}(K)$ is smaller than that of $\IP^{p_K}(K)^{1+n}$ by a factor $p_K/(n+1)+1$ (see \eqref{eq:dimTp}), leading to better accuracy per degree of freedom for the Trefftz scheme, in particular for large polynomial degrees.
The same was observed for the $p$ convergence of the method for $n=1$ in \cite[\S7.3]{SpaceTimeTDG}, both in terms of error bounds and from numerical experiments.

\subsection{The case of a first-order IBVP derived from a second-order problem: the polynomial Trefftz space \texorpdfstring{$\IWp(K)$}{Wp(K)}}
\label{s:Sp}

If the IBVP~\eqref{eq:IBVP} to be discretised is known to follow from a second-order problem, in particular the initial condition $\bsigma_0$ is a gradient, 
the slightly smaller Trefftz polynomial space $\IWp\Th$ defined in \eqref{eq:ST} can be used.
This space admits a simple basis
and enjoys the same order of convergence of $\ITp\Th$.

\subsubsection{The polynomial Trefftz space \texorpdfstring{$\IU^p(K)$}{Up(K)} for the second-order wave equation and its bases}
\label{s:Up}

We first define the space of Trefftz polynomials for the second-order wave equation: on an open $K\subset\IR^{n+1}$, with $p\in\IN_0$ and constant $c>0$,
\begin{align}\label{eq:UpK}
\IU^p(K)
:=\Big\{U\in \IP^p(K),\; -\Delta U+\frac1{c^2}\dersec{\am U} t=0\Big\}.
\end{align}
The dimension of \eqref{eq:UpK} is easily computed: the transformation $U(\bx,t)\mapsto U(\bx,t/\ri c)$ maps polynomial solutions of the wave equation to harmonic polynomials and is invertible
(if we consider complex-valued polynomials, which does not affect the space dimension).
Thus the dimension of $\IU^p(K)$ equals the dimension of the space of
harmonic polynomials of the same degree
in $n+1$ variables (cf.~\cite[eq.~(B.28)]{AndreaPhD}):
\begin{align}\label{eq:dimUp}
\dim\IU^p(K)
&=\dim\big\{P\in \IP^p(\IR^{n+1}),\; \Delta P=0\big\}\\
&=\binom{p+n-1}{n}\frac{2p+n}{p}
=\begin{cases}
2p+1 & n=1,\\
(p+1)^2 & n=2,\\
\frac{(p+1)(p+2)(2p+3)}6 & n=3.
\end{cases}\nonumber
\end{align}

The key observation to construct elements of $\IU^p(K)$ is that, for any smooth $f:\IR\to\IR$ and any unit vector $\bd\in\IR^n$, $|\bd|=1$, the space--time field $f(\bd\cdot\bx-ct)$ is solution of the wave equation.
A basis for $\IU^p(K)$ is defined by choosing suitable directions $\bd$ and polynomial functions $f$.

For each  $k=0,\ldots,p$, we fix a one-variate polynomial $Q_k\in\IP^{k}(\IR)$ of degree exactly~$k$. 
Natural choices of $Q_k$ are (scaled and/or translated) monomials, Legendre and Chebyshev polynomials.
We want to find conditions on the directions $\bd_{k,j}\in\IR^n$, with $|\bd_{k,j}|=1$, such that the following fields constitute a basis for $\IU^p(K)$
\begin{align}\label{eq:bkj}
&b_{k,j}(\bx,t):=Q_k(\bd_{k,j}\cdot\bx-ct)
, \qquad k=0,\ldots,p,\quad j=1,\ldots, \dt k,\quad \text{where}
\\[2mm]
&\dt 0:=1,\nonumber\\[-4mm]
&\dt k:=\dim\IU^k(K)-\dim\IU^{k-1}(K)
=\frac{2k+n-1}{k+n-1}\binom{k+n-1}k
=\begin{cases}
2 & n=1,\\
2k+1 & n=2,  \;k\ge1.\\
(k+1)^2 & n=3,
\end{cases}\nonumber
\end{align}
These functions (or their vector-valued analogues) are sometimes called ``transport polynomials'' \cite{PFT09,KSTW2014,KretzschmarPhD}, ``polynomial waves'' \cite{WTF14}, 
``polynomial plane waves'' \cite[eq.~(10)]{EKSTWtransparent}, ``Trefftz polynomials'' \cite{BGL2016}.
The next result gives an algebraic criterion on some matrices, which depend only on the directions $\bd_{j,k}$ and not on the specific polynomials $Q_k$, to characterise precisely which sets of directions lead to a basis.

\begin{prop}\label{prop:WaveBasis}
Fix a maximal polynomial degree $p\in\IN_0$, $Q_k\in\IP^{k}(\IR)$ of degree exactly $k$ for $k=0,\ldots,p$, and let $2\le n\in\IN$.
For $k=0,\ldots,p$, define the matrices $\bM^{(k)}\in \IC^{\dt k \times \dt k}$
$$
\bM_{\ell,m;j}^{(k)}:=Y_{\ell}^m(\bd_{k,j}) \quad 
\begin{cases}
0\le \ell\le k,\\
1\le m\le 
\binom{\ell+n-2}{\ell}\frac{2\ell+n-2}{\ell+n-2},\\
1\le j\le \dt k, 
\end{cases}$$
where 
$\{Y_\ell^m\}_{\ell\in \IN_0;\, 1\le m\le \binom{\ell+n-2}{\ell}\frac{2\ell+n-2}{\ell+n-2}}$ 
are the usual (hyper)spherical harmonics, which are orthonormal in $L^2(\{\bd\in\IR^n,|\bd|=1\})$ (see e.g.\ \cite{MUL66} or \cite[\S B.4]{AndreaPhD}).
The indices $\ell$ and $m$ identify the matrix rows, while $j$
identifies the matrix columns.

Then the set $\calS:=\{b_{k,j}, \;0\le k\le p,\; 1\le j\le \dt k\}$ defined by \eqref{eq:bkj} is a basis of $\IU^p(K)$ if and only if all the $p+1$ matrices $\bM^{(k)}$
are invertible.
\end{prop}
\begin{proof}
Since by \eqref{eq:dimUp} $\dim(\IU^p(K))=\sum_{k=0}^p\dt k$, it is enough to show that the elements of $\calS$ are linearly independent if and only if all matrices $\bM^{(k)}$, $0\le k\le p$, are invertible

\emph{Step 1.}
We first prove that the elements of $\calS_k:=\{b_{k,j}, \; 1\le j\le\dt k\}$ are linearly independent if and only if $\bM^{(k)}$ is invertible.

We denote by $c_{k,q}$, $0\le q\le k$, the coefficients of
$Q_k$, i.e.\ $Q_k(z)=\sum_{q=0}^k c_{k,q}z^q$.
To verify the independence of the elements of $\calS_k$, we set to zero a linear combination of them.
Expanding using the binomial theorem gives
\begin{align*}
0=\sum_{j=1}^{\dt k} A_j b_{k,j}(\bx,t)
&=\sum_{j=1}^{\dt k} \sum_{q=0}^k A_j c_{k,q}(\bd_{k,j}\cdot\bx-ct)^q\\
&=\sum_{j=1}^{\dt k} \; \sum_{\alpha,\beta\ge0,\; \alpha+\beta\le k}
A_j c_{k,\alpha+\beta}\binom{\alpha+\beta}\alpha   (-ct)^\alpha (\bd_{k,j}\cdot\bx)^\beta
\\
&=\sum_{\alpha=0}^k   (-ct)^\alpha  
\sum_{\beta=0}^{k-\alpha}c_{k,\alpha+\beta}\binom{\alpha+\beta}\alpha 
\sum_{j=1}^{\dt k} A_j (\bd_{k,j}\cdot\bx)^\beta
.\end{align*}
This is a polynomial in $t$, thus it is zero if and only if the coefficients of all powers of $t$ vanish:
\begin{align*}
0=\sum_{j=1}^{\dt k} A_j b_{k,j}(\bx,t) \quad \iff \quad
0&=
\sum_{\beta=0}^{k-\alpha}c_{k,\alpha+\beta}\binom{\alpha+\beta}\alpha 
\sum_{j=1}^{\dt k} A_j (\bd_{k,j}\cdot\bx)^\beta 
\quad
\forall0\le \alpha\le k.
\end{align*}
Every term in the sum over $\beta$ is a \emph{homogeneous} polynomial of degree $\beta$ in $\bx$, thus 
\begin{align*}
0=\sum_{j=1}^{\dt k} A_j b_{k,j}(\bx,t) \quad \iff \quad
0&=c_{k,\alpha+\beta}\sum_{j=1}^{\dt k} A_j (\bd_{k,j}\cdot\bx)^\beta 
\qquad\forall0\le \alpha\le k;\; 0\le \beta\le k-\alpha.
\end{align*}
Since $c_{k,k}\ne0$, otherwise $Q_k$ would have degree lower than $k$, we fix $\alpha=k-\beta$ and obtain
\begin{align*}
0=\sum_{j=1}^{\dt k} A_j b_{k,j}(\bx,t) \quad \iff \quad
0&=\sum_{j=1}^{\dt k} A_j (\bd_{k,j}\cdot\bx)^\beta 
\qquad\forall0\le \beta\le k.
\end{align*}
Now we have to deal with homogeneous polynomials only, so we can restrict ourselves to the unit sphere, i.e.\ fix $\bx=\hat\bx$, $|\hat\bx|=1$.
Denote by $P_\ell(z)=\sum_{q=0}^\ell L_{\ell,q} z^q$ the Legendre polynomial of degree $\ell$.
We can sum over $\ell$ the monomials in $\bd_{k,j}\cdot\hat\bx$ multiplying them with the coefficients $L_{\ell,q}$ and use the hyperspherical harmonic addition formula \cite[(B.29)]{AndreaPhD} (from \cite[Theorem~2]{MUL66}) to get the equivalent formula
\begin{align*}
0=\sum_{j=1}^{\dt k} A_j b_{k,j}(\bx,t) \quad \iff \quad
0=&\sum_{q=0}^\ell L_{\ell,q}\sum_{j=1}^{\dt k} A_j (\bd_{k,j}\cdot\hat\bx)^q
=\sum_{j=1}^{\dt k} A_j P_\ell(\bd_{k,j}\cdot\hat\bx)\\
=&
\sum_{j=1}^{\dt k} A_j 
\frac{\abs{\{|\bd|=1\}}}{\dH\ell} 
\sum_{m=1}^{\dH\ell} Y_{\ell}^m(\hat\bx)\conj{Y_{\ell}^m(\bd_{k,j})}\qquad \forall \ell=0,\ldots,k,
\end{align*}
where $\dH\ell:=\binom{\ell+n-2}{\ell}\frac{2\ell+n-2}{\ell+n-2}$ is the dimension of the space of homogeneous harmonic polynomials of degree $\ell$ in $n$ variables and $\abs{\{|\bd|=1\}}$ is the $(n-1)$-dimensional measure of the unit sphere in $\IR^n$.
We multiply this by $r^\ell\frac{\dH\ell}{\abs{\{|\bd|=1\}}}$ with $r>0$ and sum over~$\ell$:
\begin{align*}
0=\sum_{j=1}^{\dt k} A_j b_{k,j}(\bx,t) \quad \iff \quad
0&=
\sum_{\ell=0}^k r^\ell
\sum_{j=1}^{\dt k} A_j 
\sum_{m=1}^{\dH\ell} Y_{\ell}^m(\hat\bx)\conj{Y_{\ell}^m(\bd_{k,j})} \\
&=\sum_{\ell=0}^k 
\sum_{m=1}^{\dH\ell}
B_{\ell,m}(\hat\bx,r)
(\conj{\bM^{(k)}} \vec A)_{\ell,m} \qquad \forall r>0,
\end{align*}
where the functions in 
$\{B_{l,m}(\hat\bx,r):=r^\ell Y_{\ell}^m(\hat\bx)\}_{0\le \ell\le k;\;
1\le m\le \binom{\ell+n-2}{\ell}\frac{2\ell+n-2}{\ell+n-2}}$
are by definition linearly independent as they are an orthogonal basis
of the space of harmonic polynomials 
in $\IR^n$ (if $r$ is thought as a radial coordinate).
Thus, the expression above is zero if and only if the result of the matrix--vector product 
$(\bM^{(k)} \vec A)_{\ell,m}$ is equal to $\bzero\in \IC^{\dt k}$.
If the matrix $\bM^{(k)}$ is invertible, then $\vec A=\bzero\in
\IC^{\dt k}$ and 
we have proved that
the polynomial waves in $\calS_k$ are linearly independent.
On the other hand, if the matrix $\bM^{(k)}$ is singular, then exists
$\vec0\ne\vec A\in \ker \bM^{(k)}$, i.e.\ there exists a linear combination of polynomial waves of $\calS_k$ identically equal to zero, which is the same as saying that these are not linearly independent.

\emph{Step 2.}
We now show that the first step of the proof implies the assertion.
One implication is trivial:
given $0\le k\le p$, if the elements of $\calS$ are linearly
independent, then the same property hold for all its subsets, in
particular for $\calS_k$ and it follows that $\bM^{(k)}$ is invertible.
Now assume all matrices $\bM^{(k)}$ are invertible, thus the elements of each subsets $\calS_k$ are linearly independent.
To verify the linear independence of the elements of $\calS$ we write
\begin{align*}
0&=\sum_{k=0}^p\sum_{j=1}^{\dt k} A_{k,j} b_{k,j}(\bx,t)
=c_{p,p} 
\sum_{j=1}^{\dt p} A_{p,j} (\bd_{p,j}\cdot \bx-ct)^p
+
R_{p-1}(\bx,t)
.\end{align*}
This is the sum of a homogeneous polynomial of degree $p$ and a polynomial ($R_{p-1}$) of degree $p-1$, thus it is zero if and only if both addends are zero.
Since $\bM^{(p)}$ is invertible, step 1 of the proof with the choice $Q_p(z)=z^p$ implies that
$\{(\bd_{p,j}\cdot \bx-ct)^p\}_{1\le j\le \dt p}$ are linearly independent, thus
$A_{p,j}=0$ for all $1\le j\le \dt p$.
Repeating the argument we see that also the $A_{p-1,j}$ coefficients are zero, and proceeding by backward induction we prove that all coefficients down to $A_{0,1}$ vanish.
This is saying that the elements of $\calS$ are linearly independent.
\end{proof}

From the proof of Proposition~\ref{prop:WaveBasis}, it is also clear that a finite set $\{b_{k,j},\ j\in J\subset\{1,\ldots\,d_k\}\}$ of polynomials in the form \eqref{eq:bkj}
is linearly independent if and only if the matrix formed by the columns of $\bM^{(k)}$ with indices in $J$ has full rank.

\begin{remark}[Two-dimensional case]
For $n=2$ and $\ell>0$, the circular harmonics are
$Y_\ell^1(\cos\theta,\sin\theta)=\ee^{\ri\ell \theta}$ and
$Y_\ell^2(\cos\theta,\sin\theta)=\ee^{-\ri\ell \theta}$ thus $\bM_{\ell,m;j}^{(k)}=\ee^{(-1)^m \ri\ell\theta_{k,j}}$ 
(if $\bd_{k,j}=(\cos \theta_{k,j},\sin\theta_{k,j})$ are distinct directions), which is
product of a diagonal and a Vandermonde matrix, thus it is always invertible.
This implies that, in two space dimensions, $\calS$ is a basis of $\IU^p(K)$ for \emph{any} set of propagation directions  
$\{\bd_{k,j}, \;0\le k\le p,\; 1\le j\le\dt k\}$ with $\bd_{k,j}\ne \bd_{k,j'}$ for $j\ne j'$.
\end{remark}

\begin{remark}[Three-dimensional case]
If $n=3$, not all sets of distinct directions give linearly independent $b_{k,j}$.
An example is when too many directions are concentrated on the same circle of latitude.
Lemma 3.4.2 of \cite{AndreaPhD} describes some simple cases when $\bM^{(k)}$ are invertible (shown graphically in \cite[Fig.~1]{EKSTWtransparent}).
Moreover, Lemma 3.4.1 of \cite{AndreaPhD} ensures that ``generically'' these matrices are invertible.
Several algorithms to generate ``almost-equispaced'' unit directions in $\IR^3$ have been studied for different purposes, see e.g.\ \cite{SLW04,PTC14}. 
In particular, the directions computed in~\cite{SLW--} ensures invertibility of the matrices $\bM^{(k)}$ and upper bounds on the norms of their inverses are available, see \cite[Remark~3.4.7]{AndreaPhD}.
\end{remark}

\begin{remark}
Alternatively, one can construct a basis of $\IU^p(K)$ by evolving in time polynomial initial conditions as in \S\ref{s:TpBasis}.
The polynomial
$$
U(\bx,t)=\sum_{\malpha\in\IN_0^n,k\in \IN_0, |\malpha|+k\le p} a_{k,\malpha} \bx^\malpha t^k,
$$
with $ a_{k,\malpha}\in\IR$ belongs to $\IU^p(K)$ if and only if its coefficients satisfy the recurrence
\begin{align*}
a_{k,\malpha}&=\frac{c^2}{k(k-1)} \sum_{m=1}^n (\alpha_m+1)(\alpha_m+2)a_{k-2,\malpha+2\be_m}.
\end{align*}
Given a basis $\{\widetilde b_1,\ldots,\widetilde b_{\binom{p+n}n}\}$
of $\IP^{p}\Rn$ and a basis $\{\widehat b_1,\ldots,\widehat
b_{\binom{p-1+n}n}\}$ of $\IP^{p-1}\Rn$, one can define a basis of
$\IU^p(K)$ such that each element satisfies either
$U(\cdot,0)=\widetilde b_\ell$ and $\der Ut(\cdot,0)=0$ or
$U(\cdot,0)=0$ and $\der Ut(\cdot,0)=\widehat b_\ell$ for some $\ell$.
Again, bases constituted by monomials, Legendre or Chebyshev polynomials are allowed.
(Note that, from $\dim\IP^p\Rn=\binom{p+n}p$ and \eqref{eq:dimUp}, one has $\dim\IP^p\Rn+\dim\IP^{p-1}\Rn=\dim\IU^p(K)$.)


As in Theorem 3 of \cite{MiWi56}, the elements of this basis of $\IU^p(K)$ can also be written using iterated Laplacians as
\begin{align*}
\widetilde U_\ell(\bx,t)&=\sum_{j=0}^{\lceil \frac{p}2\rceil} 
\frac{\Delta^j \big(\widetilde b_\ell (\bx)\big) (ct)^{2j}}{(2j)!},\qquad 
\ell=1,\ldots, \binom{p+n}n,
\\
\widehat U_\ell(\bx,t)&=\sum_{j=0}^{\lceil \frac{p-1}2\rceil} 
\frac{\Delta^j \big(\widehat b_\ell (\bx)\big) (ct)^{2j+1}}{c(2j+1)!},\qquad 
\ell=1,\ldots, \binom{p-1+n}n.
\end{align*}
Equation (14) in \cite{MiWi56} (first discovered in \cite{MiWi55a}) gives an explicit formula for the case where $\widetilde b_\ell$ and $\widehat b_\ell$ are chosen as monomials.
\end{remark}
 

\begin{remark}\label{rem:NonPolyTrefftz}
The construction of the fields $b_{k,j}$ suggests possible non-polynomial discrete Trefftz spaces alternative to $\IU^p(K)$.
As proposed in \cite[\S3.1]{PFT09}, if one is interested in waves propagating with wavenumber close to a specific value $k$, then the basis functions $\sin(k(\bd\cdot\bx-ct))$ and  $\cos(k(\bd\cdot\bx-ct))$ might be used.
One can also think of using $f(\bd\cdot\bx-ct)$, with $f$ a Gaussian function or a compactly supported smooth function.
To the best of our knowledge, no non-polynomial Trefftz scheme for the time-domain wave equation has been described to date.
\end{remark}

\subsubsection{A basis of \texorpdfstring{$\IW^p(K)$}{Wp(K)}}
\label{s:SpBasis}

Bases of $\IW^p(K)$, as in \eqref{eq:ST}, can be defined as first-order derivatives of bases of $\IU^{p+1}(K)$:
\begin{align*}
\text{if} \quad \IU^{p+1}(K)=\spn\big\{b_j,\; j\in \calJ\big\}
\quad\text{then}\quad
\IW^{p}(K)=\spn\big\{(\textstyle\der{b_j}t,-\nabla b_j),\; j\in \calJ\big\}.
\end{align*}
The only elements of $\IU^{p+1}(K)$ not contributing to $\IW^p(K)$ are the constants, thus, by \eqref{eq:dimUp},
\begin{align}\label{eq:dimSp}
\dim\IW^p(K)=\dim\IU^{p+1}(K)-1
&=\binom{p+n}{n}\frac{2p+n+2}{p+1}-1
=\begin{cases}
2p+2 & n=1,\\
(p+1)(p+3) & n=2,\\
\frac{(p+2)(p+3)(2p+5)}6-1 & n=3.
\end{cases}
\end{align}
For $P_k\in\IP^k(\IR)$ of degree exactly $k$, $k=0,\ldots,p$, and for $\dt k$ as in \eqref{eq:bkj}, the set
\begin{align*}
\Big\{(v_{k,j},\bsigma_{k,j})=\big(cP_k(\bd_{k,j}\cdot\bx-ct),\bd_{k,j}P_k(\bd_{k,j}\cdot\bx-ct)\big),
\quad k=0,\ldots,p,\; j=0,\ldots,\dt{k+1}
\Big\}
\end{align*}
is a basis of  $\IW^p(K)$, provided that the directions $\bd_{k,j}$ satisfy the conditions in Proposition~\ref{prop:WaveBasis}.
(This follows immediately by choosing $Q_{k}$ such that $Q_k'=-P_{k-1}$ and $Q_0=1$.)

\begin{remark}\label{rem:TvsS}
We can characterise the polynomials that are in $\IT^p(K)$ but not in $\IW^p(K)$.
We claim that $\IT^p(K)$ can be decomposed as the direct sum of $\IW^p(K)$ and the space of the  fields $(0,\bsigma)$ where $\bsigma\in\IP^{p}(\IR^n)^n$ are time-independent, divergence-free and not gradient of harmonic polynomials of degree $p+1$.
It is clear that these fields belong to $\IT^p(K)\setminus\IW^p(K)$.
Then the dimension of the space of the divergence-free fields in $\IP^{p}(\IR^n)^n$ is $n\binom{p+n}n-\binom{p+n-1}n$, and that of its subspace of gradients of harmonic polynomials of degree $p+1$ is $\binom{p+n-1}{n-1}\frac{2p+n+1}{p+1}-1$ (see \cite[eq.~B.28]{AndreaPhD}).
Comparing with \eqref{eq:dimTp} and \eqref{eq:dimSp}, we see that $\dim\IT^p(K)-\dim\IW^p(K)=\binom{p+n}n\frac{pn-p-1}{p+1}+1$, which equals the dimension of the divergence-free polynomials in $\IP^{p}(\IR^n)^n$ that are not gradients.
(Recall that in Remark~\ref{rem:IBVP1toIBVP2} we derived a similar decomposition of the solution of a first-order acoustic wave IBVP as sum of a solution of a second-order IBVP and a 
divergence-free, time-independent field.)
\end{remark}

\begin{remark}
The bases used in previous works on Trefftz method for time-domain wave problems are related to those described so far.
In particular,
\cite[eq.~(5.7)]{SpaceTimeTDG},
\cite[eq.~(56)]{LiK16},
\cite[eq.~(13)]{PFT09}, 
\cite[eq.~(9)]{WTF14}, and 
\cite[\S5]{BGL2016} (where the approximation properties of $\IU^p(K)$ are investigated)
use the basis in \eqref{eq:bkj} with monomials as $Q_k$'s or $P_k$'s;
\cite[\S4.3]{KretzschmarPhD} and \cite[eq.~(10)]{EKSTWtransparent} adapt the same basis to the Maxwell equations.
On the other hand, \cite[Remark~3.7]{EKSW15} and \cite[\S4.2]{KretzschmarPhD} describe recurrence relations similar to \eqref{eq:RecurrenceTp} for Maxwell's equations.
\end{remark}

\subsubsection{Approximation theory for \texorpdfstring{$\IW^p(K)$}{Wp(K)}}
\label{s:SpApprox}

The following approximation result for $\IU^p(K)$ follows immediately along the lines of the proof of Corollary~\ref{cor:ApproxTp}, by applying Lemma~\ref{lem:Taylor} to the second-order wave equation, with $N=n+1$ and $N'=N''=1$.

\begin{cor}\label{cor:ApproxUp}
For $p,s\in\IN_0$, and for constant $c>0$, let $U\in H^{s+1}(K)$ be a solution of  $ -\Delta U+\frac1{c^2}\dersec U t=0$ in a domain $K\subset\IR^{n+1}$ that satisfies Assumption~\ref{as:StarShaped}.
Then, there exists $U_p\in \IU^p(K)$ such that, for all $0\le j\le  m:=\min\{p,s\}$,
\begin{align*}
\abs{U-U_p}_{\widehat H^j_c(K)}\leq
2\binom{n+j}n
\frac{(n+1)^{m+1-j}}{(m-j)!} 
\frac{h_K^{m+1-j}}{\rho^{(n+1)/2}}\abs{U}_{\widehat H^{m+1}_c(K)},
\end{align*}
where $\abs{V}_{\widehat H^k_c(K)}^2:=
\sum_{\malpha\in\IN_0^n, \alpha_t\in\IN_0,|\malpha|+\alpha_t=k}
\N{c^{1/2-\alpha_t}D^{\malpha,\alpha_t} V}_{L^2(K)}^2$ for all $V\in H^k(K)$, $k\in\IN_0$,
cf.~\eqref{eq:HjcNorm}.
\end{cor}

\begin{theorem}\label{thm:ConvergenceSp}
Under the assumptions on $(v,\bsigma)$ and $\calT_h$ stipulated in Theorem~\ref{thm:ConvergenceTp}, suppose there exists $U$ defined in $Q$ such that $v=\der U t$ and $\bsigma=-\nabla U$ in $Q$.
Let $(\Vhp,\Shp)$ be the solution of the Trefftz-DG formulation with $\bT\Th=\IWp\Th$.
Then, with $m_K:=\min\{p_K, s_K\}$,
\begin{align}\nonumber
\Tnorm{(v-\Vhp,\bsigma-\Shp)}\DG
&\le (1+C_c)\rho^{-n/2-1}
\sum_{K\in\calT_h}\xi_K^{1/2} C'_{(m_K,n)} h_K^{m_K+1/2}
\abs{(v,\bsigma)}_{H^{m_K+1}_c(K)},\\
\text{where }\quad&
C'_{(m_K,n)}=
\begin{cases}
\frac{2(n+2)(n+1)^{m_K+\frac32}}{(m_K-1)!}
& \text{if } m_K\ge1,\\
2\sqrt{n+2}(n+1)^2+\sqrt2& \text{if } m_K=0.\\
\end{cases}
\label{eq:thm:ConvergenceSp}
\end{align}
\end{theorem}
\begin{proof}
Using the norms defined in Corollary \ref{cor:ApproxUp} and \eqref{eq:HjcNorm}, for all elements $K\in\calT_h$, $k\in\IN$, $V\in H^k(K)$ we have
\begin{align}\label{eq:HjcNormComparison}
\abs{V}_{\widehat H^k_c(K)}^2\le \abs{\Big(\der Vt, -\nabla V\Big)}_{H^{k-1}_c(K)}^2
\le \min\{n+1,k\}\abs{V}_{\widehat H^k_c(K)}^2.
\end{align}
Define $U\hp\in L^2(Q)$ such that ${U\hp}_{|_K}\in\IU^{p_K}(K)$ coincides with the approximation of $U$ given by Corollary~\ref{cor:ApproxUp} with $m=m_K+1$.
If $m_K\ge1$ for all elements, we combine the previous results to obtain the assertion:
\begin{align*}
\Tnorm{(v-\Vhp,\bsigma-&\Shp)}\DG^2
\overset{\eqref{eq:QuasiOpt}}\le
(1+C_c)^2 \Tnorm{(v-\partial {U\hp}/\partial t,\bsigma+\nabla U\hp)}\DGp^2
\\&\overset{\eqref{eq:DG+<H1}}\le
(1+C_c)^2 \frac1\rho 
\sum_{K\in\calT_h}\xi_K
\bigg(\frac{n+2}{h_K}\N{\Big(v-\der {U\hp} t,\bsigma+\nabla U\hp\Big)}^2_{L^2_c(K)}\\
&\hspace{40mm}+h_K \abs{\Big(v-\der {U\hp} t,\bsigma+\nabla U\hp\Big)}^2_{H^1_c(K)}
\bigg)
\\&\overset{\eqref{eq:HjcNormComparison}}\le
(1+C_c)^2 \frac1\rho 
\sum_{K\in\calT_h}\xi_K
\bigg(\frac{n+2}{h_K}\abs{U-U\hp}^2_{\widehat H^1_c(K)}
+2h_K \abs{U-U\hp}^2_{\widehat H^2_c(K)}
\bigg)
\\&\le
(1+C_c)^2 \frac{4(n+2)^2}{\rho^{n+2}} 
\sum_{K\in\calT_h}\xi_K
\frac{(n+1)^{2m_K+3}}{(m_K-1)!^2}
h_K^{2m_K+1}\abs{U}^2_{\widehat H^{m_K+2}_c(K)},
\end{align*}
and we conclude by using again \eqref{eq:HjcNormComparison}.
The proof for the case $m_K=0$ follows similarly.
\end{proof}

\begin{rem}
Under the same assumption of Remark~\ref{rem:SimpleBoundTp} ($\alpha=\beta^{-1}=c^{-1}$, $\delta=1/2$, $\vartheta=1$, 
$\gamma=0$, 
$m_K\ge1$, $n=2$ or $n=3$),
one obtains
$C'_{(m_K,n)}\le 40 (n+1)^{m_K+1/2}/(m_K-1)!$ and
bound \eqref{eq:thm:ConvergenceSp} simplifies to
\begin{align*}
\Tnorm{(v-\Vhp,\bsigma-\Shp)}\DG
&\le \am{120}\sqrt3\rho^{-n/2-1}\sum_{K\in\calT_h} \frac{\big((n+1)h_K\big)^{m_K+1/2}}{(m_K-1)!}
\abs{(v,\bsigma)}_{H^{m_K+1}_c(K)}.
\end{align*}
\end{rem}

\section*{Acknowledgements}
\addcontentsline{toc}{section}{Acknowledgements}
The authors are grateful to Blanca Ayuso de Dios, Thomas Hagstrom, Joachim Sch\"oberl and Endre S\"uli for stimulating discussions in relation to this work.

I. Perugia has been funded by the Vienna Science and Technology Fund
  (WWTF) through the project MA14-006, and by the Austrian Science
  Fund (FWF) through the grants P~29197-N32 and F~65.

\bibliographystyle{siam}
\bibliography{references_andrea} 
\addcontentsline{toc}{section}{References}
\end{document}